\documentclass[12pt,letterpaper]{amsart}

\usepackage{amsmath, amsfonts, amssymb, amsthm, amscd}
\usepackage{graphicx, epsfig, psfrag, tikz, color}
\usepackage{hyperref}
\usepackage[english]{babel}
\usepackage{float, comment, enumitem}
\usepackage{comment}
\usepackage[margin=1in]{geometry}

\hypersetup{
	colorlinks=true,
	citecolor=green,
	linkcolor=blue,
	hypertexnames=false,
	urlcolor=cyan
}

\makeatletter
\@namedef{subjclassname@2020}{\textup{2020} Mathematics Subject Classification}
\makeatother

\newtheorem{theo}{Theorem}
\newtheorem{defi}{Definition}

\newtheorem{cons}{Construction}
\newtheorem{prop}{Proposition}
\newtheorem{coro}{Corollary}
\newtheorem{lemm}{Lemma}

\newtheorem*{theo*}{Theorem}
\newtheorem*{defi*}{Definition}
\newtheorem*{conv*}{Convention}
\newtheorem*{cons*}{Construction}
\newtheorem*{prop*}{Proposition}
\newtheorem*{coro*}{Corollary}
\newtheorem*{lemm*}{Lemma}
\newtheorem*{conj*}{Conjecture}
\newtheorem*{ques*}{Question}
\newtheorem*{prob*}{Problem}
\newtheorem*{exer*}{Exercise}
\newtheorem*{exem*}{Example}
\newtheorem*{rema*}{Remark}
\newtheorem*{claim*}{Claim}

  \def\II{\mathbb{I}}
  
 \def\NN{\mathbb{N}} 
  \def\RR{\mathbb{R}}
 \def\TT{\mathbb{T}} 
  
 \def\ZZ{\mathbb{Z}}

  \def\cF{\mathcal{F}}    \def\cH{\mathcal{H}}
  \def\cJ{\mathcal{J}}    
\def\cM{\mathcal{M}}      
  \def\cR{\mathcal{R}}    \def\cT{\mathcal{T}}
  \def\cV{\mathcal{V}}

\title[Primitive Markov Partitions]{Primitive  Geometric Markov Partitions for pseudo-Anosov Homeomorphisms}
\author{Inti Cruz}
\address{IMATE, UNAM, Oaxaca, Mexico}
\email{incruzd@gmail.com}
\date{\today}

\begin{document}

\begin{abstract}
	Let $f$ be a pseudo-Anosov homeomorphism on a closed, oriented surface. We give an effective construction of Markov partitions for $f$ based on a simple combinatorial criterion deciding when an immersed graph bounds a Markov partition. This yields an explicit algorithm: from a point $z$ at the intersection of stable and unstable separatrices of a singularity of $f$, and a sufficiently large integer $n$, it produces a partition $\mathcal{R}(f,z,n)$.
	
	Applying the algorithm to the first intersection points of $f$ we produces the set of  primitive Markov partitions. We prove the existence of an integer $n(f)$, the compatibility order of $f$, depending only on the conjugacy class of $f$, such that $\mathcal{R}(f,z,n)$ exists for all $n\ge n(f)$ and all first intersection points $z$. Each geometric Markov partition $\mathcal{R}$ has an associated geometric type $T(f,\mathcal{R})$, extending the incidence matrix; it result the geometric type is constant along orbits of primitive partitions, and for $n\ge n(f)$ the set $\mathcal{T}(f,n)$ of primitive geometric types is finite.
	By \cite{IntiThesis}, this family is canonical: two maps are topologically conjugate by an orientation-preserving homeomorphism iff they share the compatibility order  and the primitive geometric types for some $n\ge n(f)$. The types in $\mathcal{T}(f,n(f))$ are minimal and are the canonical Markov partitions of $f$.
\end{abstract}

	\maketitle

\section{Introduction}

A surface is a smooth, closed, orientable $2$-manifold. We study a distinguished family of orientation-preserving homeomorphisms of surfaces, the so-called \emph{generalized pseudo-Anosov homeomorphisms} (\cite{fathi2021thurston}). They extend the pseudo-Anosov diffeomorphisms introduced by Thurston in \cite{thurston1988geometry} by allowing $1$-prong singularities, or \emph{spines}. For brevity, we refer to them simply as pseudo-Anosov maps.  In such article Thurston outlined the classification of orientation-preserving surface diffeomorphisms up to isotopy, later extended by Handel and Thurston to orientation-preserving homeomorphisms in \cite{handelT1985new}. Their theorem states that every orientation-preserving homeomorphism of a surface is isotopic to one that is either periodic, pseudo-Anosov, or reducible. This result underlies the modern study of the mapping class group (\cite{farb2011primer}).

Pseudo-Anosov homeomorphisms play a central role in the study of surface homeomorphisms and are closely related to transitive Anosov flows \cite{fried1983transitive} and to the geometry and topology of $3$-manifolds \cite{Thurston1997}. Several constructions of generalized pseudo-Anosov homeomorphisms are known (see \cite{arnoux1981construction}, \cite{hironaka2006family}, \cite{penner1988construction}, \cite{lanneau2017tell}, \cite{Andre2004unimodal}), and further extensions have appeared more recently \cite{Andre2005extensions}, \cite{boyland2024dynamics}. 

Markov partitions constitute one of the most important tools for studying pseudo-Anosov homeomorphisms. A \emph{Markov partition} $\cR = \{R_i\}_{i=1}^n$ for a \textbf{p-A} homeomorphism $f \colon S \to S$ is a finite cover of the surface $S$ by rectangles $R_i$ (Definition~\ref{Defi: Rectangulo}) such that the image of each rectangle intersects the interior of every rectangle in the family along a horizontal sub-rectangle (Definition~\ref{Defi: Sub rectangles}). It is well known that every pseudo-Anosov homeomorphism admits a Markov partition. The classical construction is due to M.~Shub \cite[Exposition~10.5]{fathi2021thurston}, following ideas introduced by R.~Bowen in \cite{bowen1975equilibrium}. Through the existence of such Markov partitions, the main dynamical properties of pseudo-Anosov homeomorphisms are typically established by constructing a semi-conjugacy between the homeomorphism and the subshift of finite type determined by the incidence matrix of the partition.

In another direction, the 1998 monograph \cite{bonatti1998diffeomorphismes} of C.~Bonatti, R.~Langevin, and E.~Jeandenans opened a new line of research by introducing \emph{geometric Markov partitions} (Definition~\ref{Defi: Geo Markov partition}). These are Markov partitions whose rectangles are endowed with a distinguished orientation of their unstable leaves, we call it  \emph{vertical direction} of the rectangle. Their goal was to use these geometric partition to develop an algorithmic classification of pseudo-Anosov homeomorphisms up to topological conjugacy. A key ingredient in this program is the notion of the \emph{geometric type} of a geometric Markov partition (Definition~\ref{Defi: Geometric type of (f,cR)}). This invariant generalizes the classical incidence matrix: it records not only the number of intersections between the images of rectangles, but also their horizontal order and the changes of vertical orientation induced by $f$.
In \cite{CruzSubmitted}, I.~Cruz extended the ideas of Bonatti and Jeandenans \cite{bonatti1998diffeomorphismes} and established the following result:

\begin{theo}\label{Theo: Total invariant}
	A pair of pseudo-Anosov homeomorphisms admit geometric Markov partitions with the same geometric type if and only if they are topologically conjugate through an orientation-preserving homeomorphism.
\end{theo}

This theorem provides the foundation for the algorithmic and constructive classification developed in \cite{IntiThesis}.
In addition to proving that the geometric type is a complete invariant for topological conjugacy, that work also determines which abstract geometric types can actually be realized as the geometric type of a geometric Markov partition of a pseudo-Anosov homeomorphism (the \emph{realization problem}), and establishes the existence of a finite algorithm that decides when two geometric types correspond to the same pseudo-Anosov homeomorphism (the \emph{conjugacy problem}). These results will appear in forthcoming articles.

The objective of this article is to introduce a distinguished class of Markov partitions, which we call \emph{primitive Markov partitions}, and to establish their main properties. We will present our results at the end of this introduction; however, we first outline the principal steps and advantages of our construction in order to highlight the significance of this family.

\subsection{Presentation of results}

Our first result is a criterion ensuring that an immersed graph in $S$ forms the boundary of a Markov partition; it is stated in Proposition~\ref{Prop:Compatibles implies Markov partition} and proved at the beginning of Section~\ref{Sec: Contruccion Particion Markov}. Graphs adapted to $f$ are defined in Definition~\ref{Defi: Adapted graph}, and compatible graphs in Definition~\ref{Defi: Compatible graphs}. A $\delta^{s}$-graph consists of stable arcs, each contained in a single stable separatrix, with one endpoint at a singularity and the other at its intersection with an unstable separatrix. Two graphs are compatible when every endpoint of one lies in the other.

A Markov partition $\cR$ is  \emph{adapted} to $f$ (Definition~\ref{Defi: Partion wellsuited/corner/adapted partition}) when the only periodic points on $\partial\cR$ are singularities, and each such singularity occurs at a corner of any rectangle containing it. These are particularly convenient, since singularities of a \textbf{p-A} homeomorphism always lie on the boundary, though not necessarily at corners.

\begin{prop*}[\ref{Prop:Compatibles implies Markov partition}]
	Let $\delta^u$ and $\delta^s$ be graphs adapted to $f$ and compatible. Then the closure in $S$ of each connected component of
	\[
	\overset{o}{S} := S \setminus \left( \bigcup \textbf{Ex}^s(\delta^u) \,\cup\, \bigcup \textbf{Ex}^u(\delta^s) \right)
	\]
	is a rectangle whose stable (resp.\ unstable) boundary is contained in $\bigcup \delta^s$ (resp.\ $\bigcup \delta^u$). Moreover, the family $\mathcal{R}(\delta^s,\delta^u)$ of these rectangles is an adapted Markov partition for $f$.
\end{prop*}

We then apply the criterion of Proposition~\ref{Prop:Compatibles implies Markov partition} to obtain the algorithmic Construction~\ref{Cons: Recipe for Markov partitions} of adapted Markov partitions. The construction begins with a single point  $z \in F^s_p \cap F^u_q$, the intersection of the stable and unstable separatrices of periodic points $p$ and $q$, and proceeds through a sequence of elementary steps that are easily visualized. A key feature of this approach—beyond its simplicity and explicit geometric nature—is that the resulting Markov partition depends only on the initial point $z$ (Corollary~\ref{Coro: Existence adapted Markov partitions}) and on a single integer parameter $n\in\NN$ chosen larger than a threshold $n(z)\in\NN$, the \emph{compatibility coefficient} of $z$ (see Construction~\ref{Cons: Recipe for Markov partitions}, Item~5).

Section~\ref{Sec: Particiones primitivas} applies Construction~\ref{Cons: Recipe for Markov partitions} to a distinguished family of points of a \textbf{p-A} homeomorphism $f$, the \emph{first intersection points} (Definition~\ref{Defi: first intersection points}).  
For singularities $p$ and $q$ with stable and unstable separatrices $F^s_p$ and $F^u_q$, a point  
$z \in F^s_p \cap F^u_q$ is a first intersection point if $(p,z]^s \cap (q,z]^u = \{z\}$.

Proposition~\ref{Prop: Finite number of first intersection points} shows that $f$ admits only finitely many orbits of first intersection points.  
Corollary~\ref{Coro: n(f) number} further states that each such point $z$ carries a \emph{compatibility order} $n(z)\in\NN$ such that, for every $n\ge n(z)$, Construction~\ref{Cons: Recipe for Markov partitions} produces a Markov partition of order $n$, called a \emph{primitive Markov partition} of $f$ (Definition~\ref{Defi: Primitive Markov partitions}). These partitions depend solely on $z$.

Finally, Proposition~\ref{Prop: Conjugates then primitive Markov partition} establishes the following fundamental property: if $f$ and $g$ are pseudo-Anosov homeomorphisms that are topologically conjugate via $h$, then $h$ sends every primitive Markov partition of $f$ to a primitive Markov partition of $g$. In this sense, the family of primitive Markov partitions is canonical.

\begin{prop*}[\ref{Prop: Conjugates then primitive Markov partition}]
	Let $f: S_f \to S_f$ and $g: S_g \to S_g$ be \textbf{p-A} homeomorphisms that are topologically conjugate via a homeomorphism $h: S_f \to S_g$. Let $z \in S_f$ be a first intersection point of $f$. Then:
	\begin{itemize}
		\item[i)] $h(z)$ is a first intersection point of $g$.
		\item[ii)] The compatibility coefficients coincide: $n(z) = n(h(z))$.
		\item[iii)] For every $n \ge n(z)=n(h(z))$, the image under $h$ of the primitive Markov partition of $f$ generated by $z$ and of order $n$ is the primitive Markov partition of $g$ generated by $h(z)$ and of order $n$, i.e.
		$$
		h(\cR(z,n)) = \cR(h(z),n).
		$$
	\end{itemize}
\end{prop*}

Although a \textbf{p-A} homeomorphism $f$ admits infinitely many Markov partitions, these are in general incomparable and lack any natural ordering or structural hierarchy, making their incidence matrices difficult to analyze.  
In contrast, the family of primitive Markov partitions satisfies strong finiteness and ordering properties, as stated in the following corollary.

\begin{coro*}[\ref{Coro: Finite orbits of primitive Markov partitions}]	
	Let $n\geq n(f)$. Then, there exists a finite but nonempty set of orbits of primitive Markov partitions of order $n$.
\end{coro*}

Given a rectangle $R$, one can choose orientations of the stable and unstable leaves of $\cF^s$ and $\cF^u$ through $R$ compatible with the surface orientation. The induced orientation on the unstable foliation is called the \emph{vertical direction} of $R$ (Definition~\ref{Defi: Geometric rec}).  
A Markov partition whose rectangles carry vertical directions is called a \emph{geometric Markov partition}.  
In analogy with the classical incidence matrix of a Markov partition $(f,\cR)$, the \emph{geometric type} of a geometric Markov partition records not only the number of intersections between rectangles but also their horizontal order and the change of vertical direction under $f$ (Definition~\ref{Defi: Geometric type of (f,cR)}).

Geometric types were introduced by C.~Bonatti and R.~Langevin \cite{bonatti1998diffeomorphismes} in their study of invariant neighborhoods of saddle-type uniformly hyperbolic sets for $C^1$-structurally stable surface diffeomorphisms (\emph{Smale diffeomorphisms}; see \cite{hasselblattkatok2002handbook}).  
This classification program was completed by F.~Béguin in \cite{beguin2002classification,beguin2004smale}.  
Adapting these ideas, we prove our main result, Theorem~\ref{Theo: finite geometric types}, asserting that for every $n$ there are only finitely many primitive geometric types of order $n$.

\begin{theo*}[\ref{Theo: finite geometric types}]
	For every $n \ge n(f)$, the set
	$$
	\cT(f,n) := \{\,T(f,\cR) : \cR \in \cM(f,n)\,\},
	$$
	of \emph{primitive geometric types} of $f$ of order $n$ is finite.
\end{theo*}

Thus, the \emph{canonical} Markov partitions of $f$ are the elements of $\cM(f,n(f))$, and the corresponding \emph{canonical geometric types} form the finite set $\cT(f,n(f))$.

In \cite{IntiThesis} it was shown that two pseudo-Anosov homeomorphisms are topologically conjugate by an orientation-preserving homeomorphism if and only if they admit geometric Markov partitions of the same geometric type. We therefore obtain:

\begin{coro}\label{Coro: same total invariants}
	Two pseudo-Anosov homeomorphisms $f$ and $g$ are topologically conjugate via an orientation-preserving homeomorphism if and only if $n(f)=n(g)$ and, there exist  $n \ge n(f)=n(g)$ such that their primitive geometric types of order $n$ coincide:
	$$
	\cT(f,n)=\cT(g,n).
	$$
\end{coro}

\subsection{Remarks and Perspectives for Further Research}

For a saddle-type basic piece $K$ of a $C^{1}$ structurally stable surface diffeomorphism $\phi$, F.~Béguin shows in \cite[Proposition~1]{beguin2004smale} that the integer $n(\phi,K)$ (corresponding to our $n(f)$) satisfies the upper bound
$$
n(\phi,K)\le 142\, n_T^{\,2},
$$
where $T$ is the geometric type of a geometric Markov partition $\mathcal{R}$ of $K$, and $n_T$ denotes the number of rectangles in $\mathcal{R}$. Moreover, if at least one geometric type realized by a geometric Markov partition of $K$ is known, \cite[Proposition~2]{beguin2004smale} guarantees the existence of an algorithm that computes all primitive geometric types of $(\phi,K)$ of order $n$ for every $n\ge n(\phi,K)$. In particular, the family of primitive geometric types is algorithmically computable.

In \cite{IntiThesis} we developed an algorithm that decides whether two geometric types arise from the same pseudo-Anosov homeomorphism. Although effective, this procedure requires an \emph{exponential} number of steps with respect to the number of rectangles of the underlying Markov partitions, as it relies on Béguin's algorithm for saddle-type basic pieces—an algorithm that must handle situations that never occur for geometric Markov partitions of 
pseudo-Anosov homeomorphisms. This observation motivates the development of a direct and explicit algorithm which, given
$$
T \in \mathcal{T}(f,n),
$$
computes the primitive geometric types of the adjacent orders,
$$
\mathcal{T}(f,n-1) 
\qquad\text{and}\qquad 
\mathcal{T}(f,n+1).
$$
We expect that such a constructive framework will lead to a substantially more efficient algorithm for the conjugacy problem in the pseudo-Anosov setting.

\section{Geometric Markov partition and their geometric types}\label{Sec: Geo part Tipes}

We adopt the following definition, which includes the case of Anosov diffeomorphisms on the $2$-dimensional torus $\TT^2$.

\begin{defi}[\textbf{p-A} homeomorphism]\label{Defi: Homeo p-A}
	Let $S$ be a surface. An orientation-preserving homeomorphism $f: S \to S$ is a \emph{generalized pseudo-Anosov homeomorphism}, or \textbf{p-A} homeomorphism for short, if there exist $\lambda > 1$ and two $f$-invariant transverse measured foliations, which may contain spines: $(\cF^s, \mu^s)$ and $(\cF^u, \mu^u)$, such that:
	\begin{equation*}
		f_*(\mu^s) = \lambda \mu^s \quad \text{and} \quad f_*(\mu^u) = \frac{1}{\lambda} \mu^u.
	\end{equation*}
	
	We call $\lambda$ the \emph{stretch factor} of $f$, and we refer to $(\cF^s, \mu^s)$ and $(\cF^u, \mu^u)$ as the \emph{stable} and \emph{unstable foliations} of $f$, respectively. If the \emph{invariant foliations} of $f$ have no singularities, we assume that the set of singular points coincides with a finite family of periodic points of $f$.
\end{defi}

For a detailed definition of measured foliations and basic results on \textbf{p-A} homeomorphisms, the reader may refer to \cite[Sections 11 and 14]{farb2011primer} or \cite[Expositions 5, 10, and 12]{fathi2021thurston}. It is well known that every \textbf{p-A} homeomorphism is, in fact, a diffeomorphism away from its singularities, and its dynamics are uniformly hyperbolic outside of them. This provides a useful framework for visualizing our discussion.

\subsection{Geometric Markov partitions}
 We introduce definitions and basic results on geometric Markov partitions and their geometric types for \textbf{p-A} homeomorphisms. The primary reference for this subject is \cite{bonatti1998diffeomorphismes}, where the theory was developed for Smale diffeomorphisms. In particular, the authors consider Markov partitions with disjoint rectangles, which requires us to adapt their concepts to our setting.

Let $\emptyset \neq r \subset S$ be an open and connected set. For every $x \in r$, let $F_x^s \in \cF^s$ and $F_x^u \in \cF^u$ denote the stable and unstable leaves passing through $x$, respectively. Define $\overset{o}{I}_x \subset F_x^s$ as the unique connected component of $F_x^s \cap r$ containing $x$, and $\overset{o}{J}_x \subset F_x^u$ as the unique connected component of $F_x^u \cap r$ containing $x$. We introduce the following notation: $I_x := \overline{\overset{o}{I}_x}$ and $J_x := \overline{\overset{o}{J}_x}$.

\begin{defi}\label{Defi: Bifoleado}
	An open and connected set $\emptyset \neq r \subset S$ is \emph{trivially bi-foliated} by $\cF^s$ and $\cF^u$ if, for every $x, y \in r$, $I_x$ and $J_y$ are homeomorphic to closed intervals and $I_x \cap J_y = \{z(x, y)\}$, where $z(x, y) \in r$.
\end{defi}

\begin{defi}[Rectangle]\label{Defi: Rectangulo}
	A rectangle $R \subset S$ for $f$ is the closure $R = \overline{r}$ of a non-empty open and connected set $r \subset S$, which is trivially bi-foliated by $\cF^s$ and $\cF^u$.
\end{defi}

If the homeomorphism $f$ has been previously specified, for brevity, we simply say that $R$ is a rectangle instead of referring to it as a rectangle for $f$. Take a bi-foliated set $r$; it is essentially as shown in the figure below. Let $x \in r$, and consider the closed interval $I_x$ with endpoints $x_+$ and $x_-$. There are two possibilities for $x_+$ (and for $x_-$ as well): either it is the endpoint of a single interval $I_x$, or there exists $y \notin \overset{o}{I_x}$ such that $I_y$ has $x_+$ as an endpoint.  
\begin{figure}[ht]
	\centering
	\includegraphics[width=0.5\textwidth]{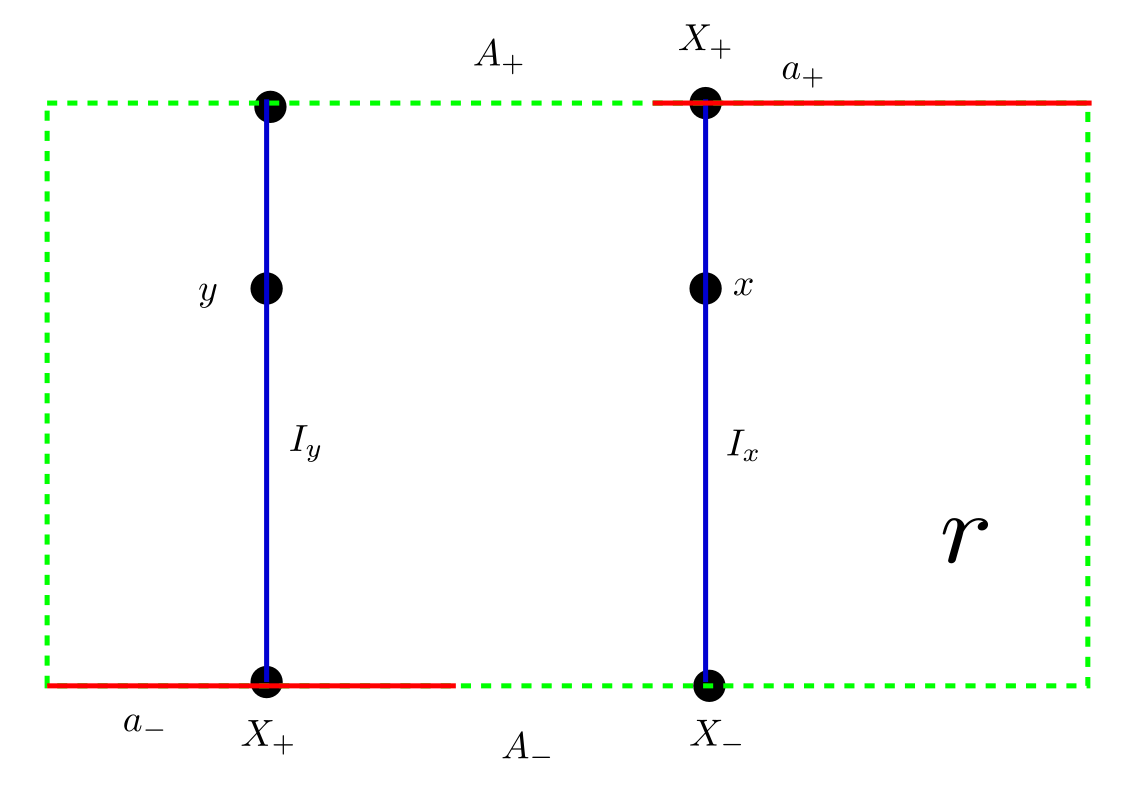}
	\caption{Boundary identification in $R=\overline{r}$}
	\label{Fig: Identify boundaries rec}
\end{figure}

Following the notation in Figure \ref{Fig: Identify boundaries rec}, the previous argument implies a certain identification of points in $A_+$ with points in $A_-$. In fact, there must be an open interval $a_+ \subset A_+$ and $a_- \subset A_-$ that are identified. However, since there is no closed leaf in the unstable foliation of $f$, the intervals $a_-$ and $a_+$ cannot be equal to the respective $A_-$ or $A_+$.  This implies that our rectangles must be isomorphic to some of the following examples:

\begin{figure}[ht]
	\centering
	\includegraphics[width=0.5\textwidth]{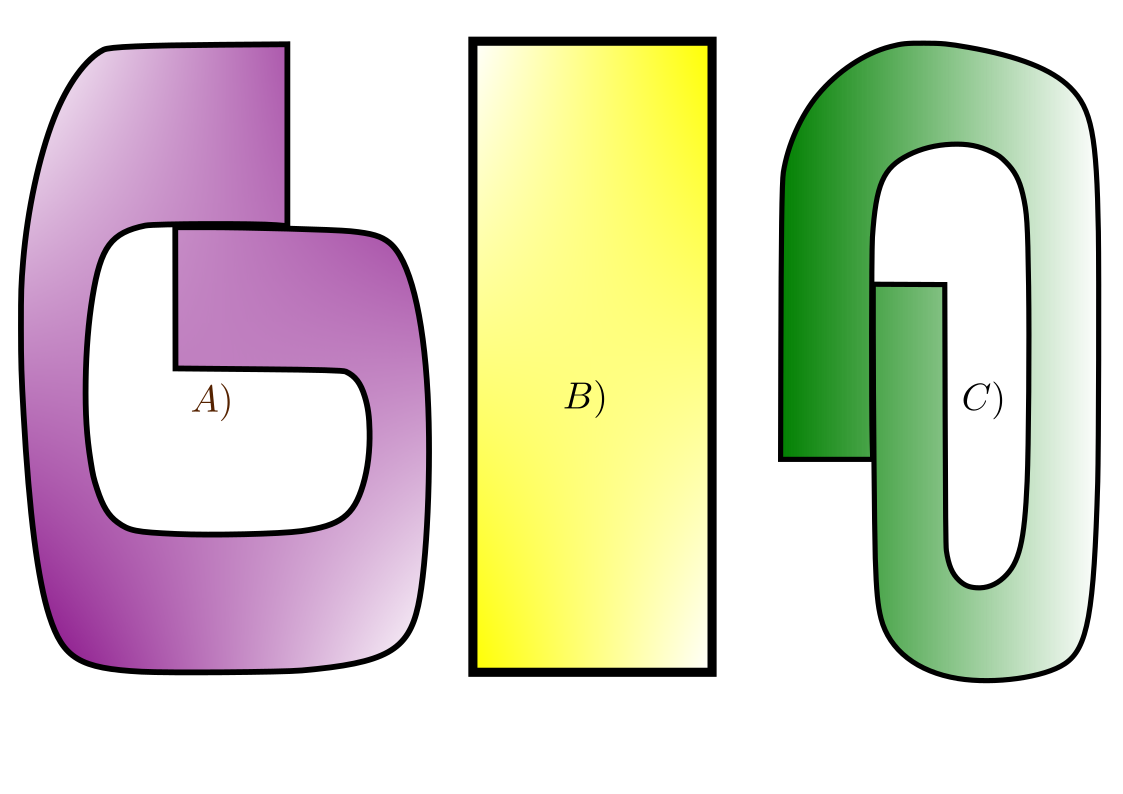}
	\caption{Rectangles $A$ and $C$ with boundary identifications; rectangle $A$ is embedded.}
	\label{Fig: Different types of rectangles}
\end{figure}

If $R$ is a rectangle, the \emph{topological interior} of $R$ as a subspace of the surface $S$ is denoted by $\text{Int}(R)$. It is not difficult to see that $\overset{o}{R} \subset \text{Int}(R)$, but $\overset{o}{R}$ will be more relevant for our discussions. Let $x \in \overset{o}{R}$. Consider the following convention: $\overset{o}{I}_x(R)$ is the only connected component of $\overset{o}{R} \cap F^s_x$ that contains $x$. The open interval $\overset{o}{J}_x(R) \subset \overset{o}{R} \cap F^u_x$ is similarly defined.

\begin{defi}[Subrectangle]\label{Defi: Sub rectangles}
	Let $R \subset S$ be a rectangle. A rectangle $H \subset R$ is a \emph{horizontal subrectangle} of $R$ if, for every $x \in \overset{o}{H}$, $\overset{o}{I}_x(H) = \overset{o}{I}_x(R)$. Similarly, a rectangle $V \subset R$ is a \emph{vertical subrectangle} of $R$ if, for every $x \in \overset{o}{V}$, $\overset{o}{J}_x(V) = \overset{o}{J}_x(R)$.
\end{defi}

We must use the orientations of the surface $S$ and the standard orientation of the Euclidean plane $\RR^2$ to induce a vertical and horizontal orientation for each rectangle on the surface. We must to introduce the notion of  parametrization of a rectangle.

\begin{lemm}\label{Lemm: Rec parametrizado}
	Let $R = \overline{r}$ be a rectangle for $f$. Then there exists a continuous immersion $\rho: [0,1] \times [0,1] \to S$ with the following properties:
	\begin{enumerate}
		\item The image of $\rho$ is the rectangle $R$, and we call  \emph{interior} of the rectangle $R$ to the set $\overset{o}{R} := \rho((0,1) \times (0,1))$.
		
		\item The restriction $\rho: (0,1) \times (0,1) \to \overset{o}{R}$ is an orientation-preserving homeomorphism.
		
		\item For every $t \in [0,1]$, the arc $I_t := \rho([0,1] \times \{t\})$ is contained in a unique stable leaf of $\mathcal{F}^s$, and the restriction of $\rho$ to $[0,1] \times \{t\}$ is a homeomorphism onto its image. We call each arc $I_t$  \emph{horizontal leaf} of $R$.
		
		\item For every $s \in [0,1]$, the arc $J_s := \rho(\{s\} \times [0,1])$ is contained in a unique unstable leaf of $\mathcal{F}^u$, and the restriction of $\rho$ to $\{s\} \times [0,1]$ is a homeomorphism onto its image. We call each arc $J_s$  \emph{vertical leaf} of $R$.
		
	\end{enumerate}
	A map such as $\rho$ is called a \emph{parametrization} of $R$.
\end{lemm}

\begin{proof}
	We shall to construct a homeomorphism $\rho: (0,1) \times (0,1) \to S$ such that $\rho((0,1) \times (0,1)) = r$, and satisfies the last three items in the proposition. Subsequently, we must to extend $\rho$ continuously to $[0,1] \times [0,1] \to S$ in such a way that $\rho([0,1] \times [0,1]) = \overline{D} = R$, while satisfy all the other properties to be a parametrization of $R$.
	
	Let $x \in r$ be a distinguished point, and let $y \in r$. Since $r$ is trivially bi-foliated by $\cF^s$ and $\cF^u$, there exists a unique point $y_1$ in the intersection $\overset{o}{I_x} \cup \overset{o}{J_y}$ and a unique $y_2$ in $\overset{o}{J_x} \cup \overset{o}{I_y}$.  Following the notation in Fig. \ref{Fig: Arcos para la parametrizacion}  let $\epsilon(y) = 1$ if $\overset{o}{I_y}$ intersects the separatrix $J_1$ of $x$, and $-1$ otherwise. Similarly, $\delta(y) = 1$ if $\overset{o}{J_y}$ intersects the separatrix $I_1$ of $x$, and $-1$ otherwise.  $\phi: r \to \RR^2$ which assigns to  $y\in 2$ the point:
	$$
	(\epsilon(y)\cdot\mu^u([x,y_1]^s),\delta(y)\cdot \mu^s([x,y_2]^u)) \in  \RR^2.
	$$
	
	\begin{figure}[ht]
		\centering
		\includegraphics[width=0.4\textwidth]{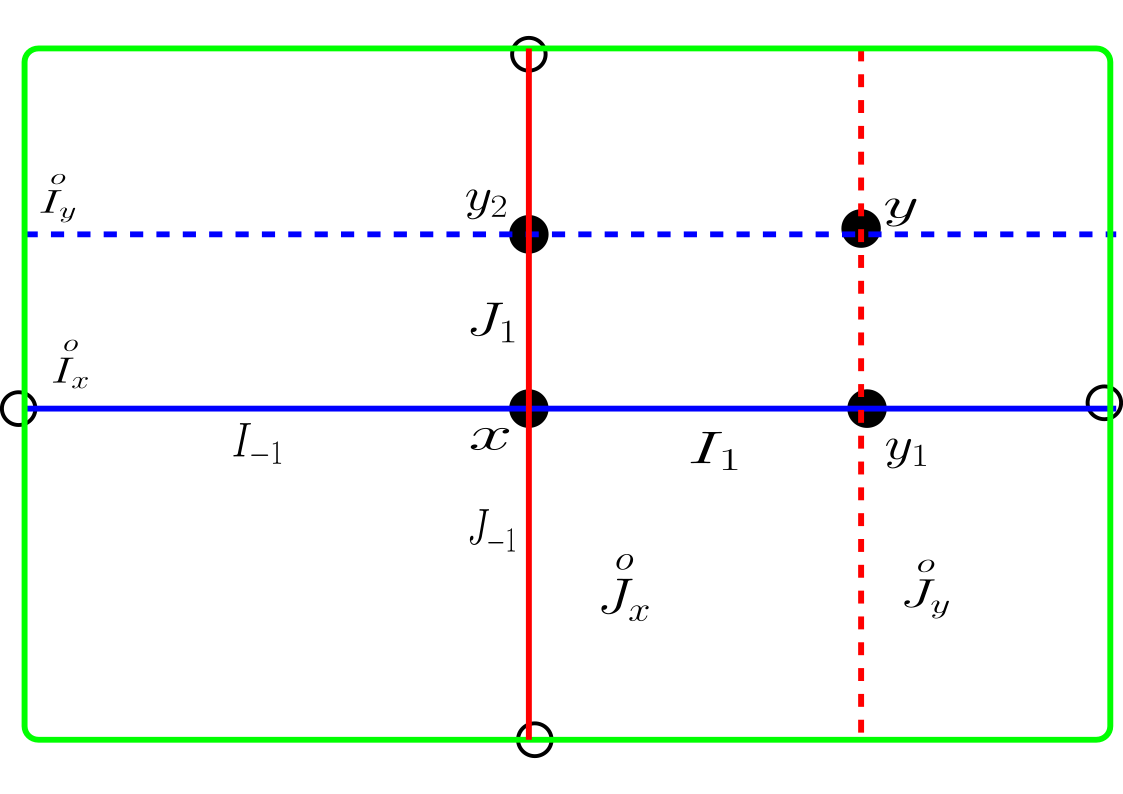}
		\caption{The arcs $[x,y_1]^s$ and $[x,y_1]^u$. }
		\label{Fig: Arcos para la parametrizacion}
	\end{figure}
	
	Since both measures are absolutely continuous with respect to Lebesgue and non-atomic, we can conclude that $\phi$ is a continuous homeomorphism, whose inverse is also continuous. Furthermore, since $r$ is trivially bi-foliated by $\cF^s$ and $\cF^u$ and the transverse measures $\mu^s$ and $\mu^u$ are invariant by isotopy on the leaves, the measures of any two arcs $I_x := \overline{\overset{o}{I_x}}$ and $I_y := \overline{\overset{o}{I_y}}$ have the same transverse measure, $\mu^u(I_x) = \mu^u(I_y)$. Similarly, the arcs $J_x := \overline{\overset{o}{J_x}}$ and $J_y := \overline{\overset{o}{J_y}}$ satisfy $\mu^s(J_x) = \mu^s(J_y)$. This means that 
	
	$$\phi(r):= (A,B) \times (C,D):=\overset{o}{H}\subset\RR^2.$$ 
	
	Now we can define $\rho' := \phi^{-1}: \overset{o}{H} \to R$ and compose it with an affine transformation that preserves or changes the orientation of $\mathbb{R}^2$ in order to obtain a function $\rho: (0,1) \times (0,1) \to S$ that satisfies items $2), 3), 4)$ of our statement.  
	It is not difficult to see that such a function can be continuously extended to $\rho: [0,1] \times [0,1] \to S$, satisfying all the required conditions. This concludes our proof.
	
\end{proof}

There is not a unique parametrization for a rectangle so we impose the following equivalence.

\begin{defi}\label{Defi: Parametrizaciones equiv}
	Let $\rho_1$ and $\rho_2$ be two parametrizations of the rectangle $R\subset S$. They are equivalent if the transition map between $\rho_1$ and $\rho_2$, restricted to the interior of the rectangle, $\overset{o}{R}$, is given by:
	\begin{equation}\label{Equa: cambio de coordenadas parametrizaciones}
		\rho_2^{-1} \circ \rho_1 := (\varphi_s, \varphi_u) : (0,1) \times (0,1) \rightarrow (0,1) \times (0,1),
	\end{equation}
	where $\varphi_s, \varphi_u: (0,1) \to (0,1)$ are increasing homeomorphisms.
\end{defi}

Since the parametrizations we consider are orientation-preserving, it is clear that there is only one equivalence class of parametrizations for a rectangle. Moreover, if  $\rho_1$ and $\rho_2$ are equivalent parametrizations of a rectangle $R$ and  the interiors of some horizontal leaves of $R$, given by  $\overline{\overset{o}{I_1}} = \rho_1((0,1) \times \{t\})$ and $\overline{\overset{o}{I_2}} = \rho_1((0,1) \times \{s\})$, intersect, then  $I_1=I_2$. This property allows us to introduce the vertical and horizontal foliations of $R$ independently of its parametrization.

\begin{defi}[Horizontal and vertical foliations]\label{Defi: Foliacion vertical y horizontal de Rec}
	Let $R$ be a rectangle, and let $\rho: [1,1]\times [0,1] \to R \subset S$ be any parametrization of $R$.  
	The \emph{horizontal foliation} of $R$ is the family of stable arcs $\cH(R) = \{I_t\}_{t \in [0,1]}$, where $I_t$ is as described in item $(3)$ of Lemma \ref{Lemm: Rec parametrizado}.  
	
	Similarly, the \emph{vertical foliation} of $R$ is the family of unstable arcs $\cV(R) = \{J_s\}_{s \in [0,1]}$, where $J_s$ is as described in item $(4)$ of Lemma \ref{Lemm: Rec parametrizado}.
\end{defi}

The parametrizations are used now to give a local orientations to our rectangles.

\begin{defi}[Geometric rectangle]\label{Defi: Geometric rec}
	Given a parametrization $\rho: [0,1] \times [0,1] \to S$ of the rectangle $R \subset S$, a \emph{vertical direction} in $R$ is defined by selecting an orientation $\xi_v$ for the vertical lines in $[0,1]\times [0,1]$ and then we use $\rho$ to induce an orientation on the vertical leaves of $R$. 
	
	A rectangle with a fixed vertical direction is called  \emph{geometric rectangle}.
\end{defi}

Fixing a vertical direction $\xi_v$ in $[0,1]\times [0,1]$, there is a unique horizontal direction $\xi_h$ in $[0,1]\times [0,1]$ such that the oriented frame $(\xi_v, \xi_h)$ coincides with the standard orientation of $\mathbb{R}^2$ (see Fig. \ref{Fig: Orientaciones verticales en II2}). This allows us to introduce the following definition.

\begin{defi}\label{Defi: Horizontal Vertical direction Rec}
	Given a parametrization $\rho:[0,1] \times [0,1] \to R \subset S$ and a vertical direction on $[0,1] \times [0,1]$, along with its corresponding horizontal direction, $\rho$ induces an orientation on the stable and unstable leaves passing through $R$. These orientations define the vertical and horizontal directions in the vertical and horizontal foliations, $\cH(R)$ and $\cV(R)$, of the rectangle $R$.
\end{defi}

\begin{defi}\label{Defi: Sub rec geometrico}
	Let $R$ be a geometrized rectangle, and let $H$ and $V$ be horizontal and vertical sub-rectangles of $R$. The \emph{canonical geometrization} of $H$ assigns it the same vertical and horizontal directions as those of $R$.
\end{defi}

\begin{figure}[h]
	\centering
	\includegraphics[width=0.4\textwidth]{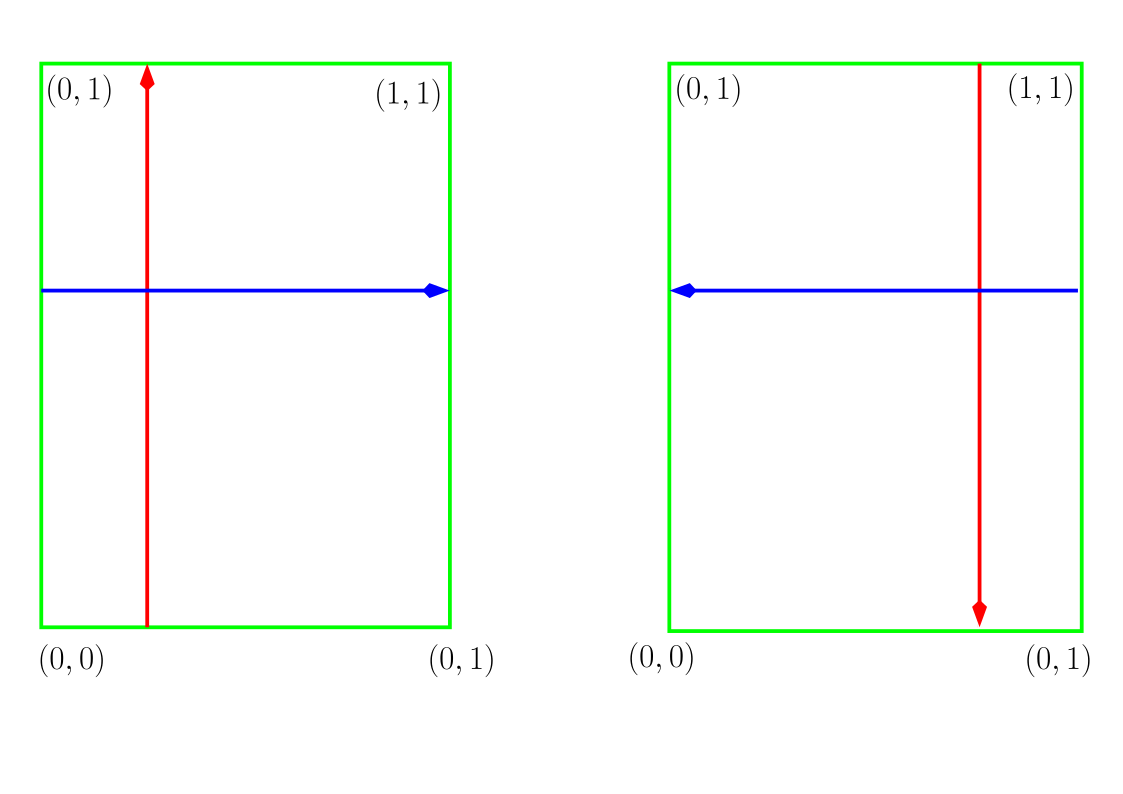}
	\caption{The vertical directions in $[0,1]\times [0,1]$}
	\label{Fig: Orientaciones verticales en II2}
\end{figure}

\begin{defi}[Geometric Markov partition]\label{Defi: Geo Markov partition}
	Let $f: S \rightarrow S$ be a \textbf{p-A} homeomorphism. A \emph{Markov partition} for $f$ is a family of labeled rectangles $\mathcal{R} = \{ R_i \}_{i=1}^n$ satisfying the following properties:
	\begin{enumerate}
		\item The surface $S$ is the union of these rectangles: $S = \cup_{i=1}^n R_i$.
		
		\item Their interiors are disjoint, i.e., for all $i \neq j$,  $\overset{o}{R_i} \cap \overset{o}{R_j} = \emptyset$.
		
		\item For every $i, j \in \{1, \dots, n\}$, the closure of each non-empty connected component of $\overset{o}{R_i} \cap f^{-1}(\overset{o}{R_j})$ is a horizontal subrectangle of $R_i$.
		
		\item For every $i, j \in \{1, \dots, n\}$, the closure of each non-empty connected component of $f(\overset{o}{R_i}) \cap \overset{o}{R_j}$ is a vertical subrectangle of $R_j$.
	\end{enumerate}
	If, in addition, every rectangle in the family $\cR$ is a geometric rectangle, we say that $\cR$ is a \emph{geometric Markov partition} of $f$. In this case, we use the notation $(f, \cR)$ to indicate that $\cR$ is a geometric Markov partition for the map $f$.
\end{defi}

\begin{figure}[h]
	\centering
	\includegraphics[width=0.5\textwidth]{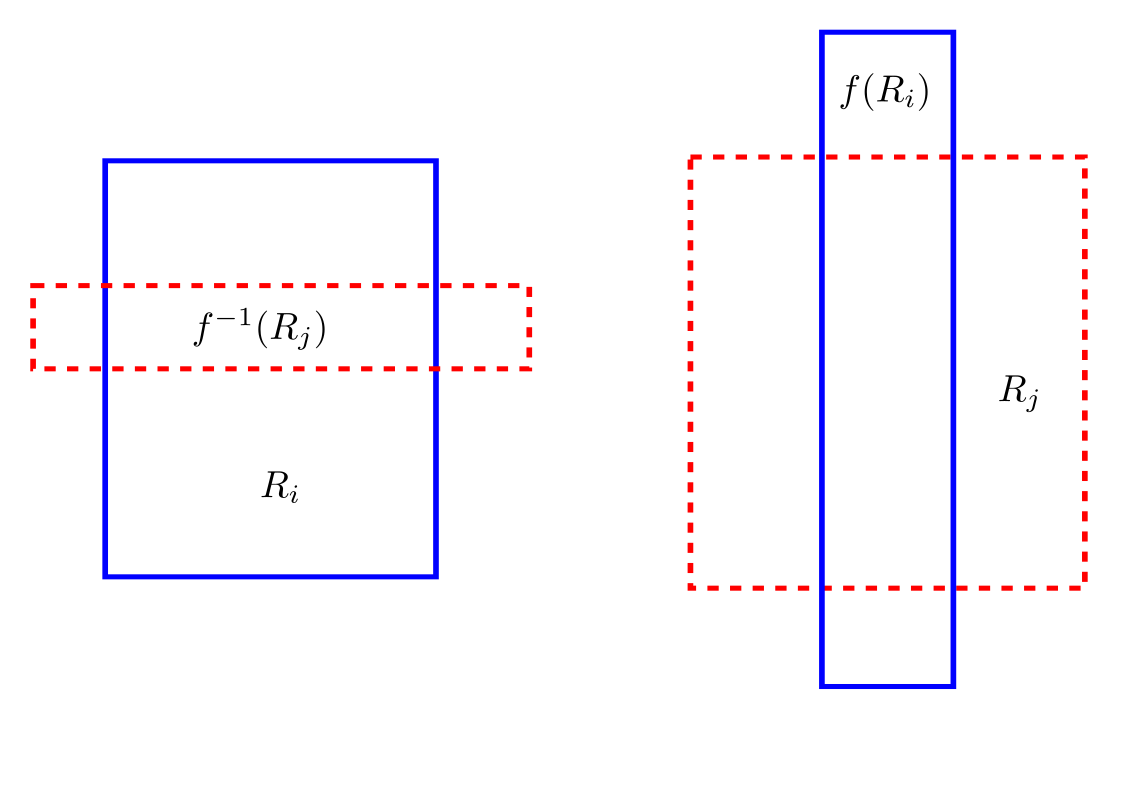}
	\caption{Item $(3)$ and $(4)$ in Definition \ref{Defi: Geo Markov partition}}
	\label{Fig: Markov properti}
\end{figure}

\begin{defi}\label{Defi: Ver Hor sub rectangle (f,R)}
	Let $f$ be a \textbf{p-A} homeomorphism and let $\cR=\{R_i\}_{i=1}^n$ be a geometric Markov partition. The set of horizontal subrectangles of $(f,\cR)$ contained in $R_i$ is
	$$
	\cH(f,\cR,i) = \{ \overline{H} : H \text{ is a connected component of } \overset{o}{R_i} \cap \cup_{j=1}^n \overset{o}{R_j} \}.
	$$ 
	The set of vertical subrectangles of $(f,\cR)$ contained in $R_j$ is given by
	$$
	\cV(f,\cR,j) = \{ \overline{V} : V \text{ is a connected component of } \overset{o}{R_j} \cap \cup_{i=1}^n f^{-1}(\overset{o}{R_i}) \}.
	$$ 
	In this manner, the sets of horizontal and vertical subrectangles of $(f,\cR)$, respectively, are given by:
	$$
	\cH(f,\cR) := \cup_{i=1}^n \cH(f,\cR,i) \, \text{ and } \, \cV(f,\cR) := \cup_{j=1}^n \cV(f,\cR,j).
	$$
\end{defi}

We must introduce some  interesting subsets inside a rectangle.

\begin{defi}\label{Defi: L-R sides}
	Let $R$ be a geometric rectangle for $f$, and let $\rho:\II^2\to R$ be a parametrization of $R$. We have the following distinguished sets:
	\begin{itemize}
		\item The \emph{left} and \emph{right} sides of $R$ are given by $\partial^u_{-1}R := \rho(\{0\} \times [0,1])$ and $\partial^u_{1}R := \rho(\{1\} \times [0,1])$, respectively. Each of these arcs is an $s$-\emph{boundary component} of $R$.
		
		\item The \emph{lower} and \emph{upper} sides of $R$ are defined as $\partial^s_{-1}R := \rho([0,1] \times \{0\})$ and $\partial^s_{+1}R := \rho([0,1] \times \{1\})$, respectively. Each of these arcs is a $u$-\emph{boundary component} of $R$.
		
		\item The \emph{horizontal boundary} of $R$ is defined as $\partial^s R := \partial^s_{-1}R \cup \partial^s_{+1}R$, and the \emph{vertical boundary} of $R$ as $\partial^u R := \partial^u_{-1}R \cup \partial^u_{+1}R$.
		
		\item The \emph{boundary} of $R$ is $\partial R := \partial^s R \cup \partial^u R$.
		
		\item The \emph{corners} of $R$ are the points given by the image of $\rho$ at $(s,t) \in \II^2$, where each corner is labeled $C_{s,t} = \rho(s,t)$.
		
		\item For all $x \in \overset{o}{R}$, $I_x$ denotes the \emph{horizontal leaf} of $\mathcal{I}(R)$ passing through $x$, and $J_x$ denotes the \emph{vertical leaf} of $\mathcal{J}(R)$ passing through $x$.
	\end{itemize}
\end{defi}

In the same spirit, we need to assign names to certain distinguished points and subsets within a geometric Markov partition that highlight their position inside a rectangle.
\begin{defi}\label{Defi: Boundary points Markov partition}
	Let $\mathcal{R} = \{R_i\}_{i=1}^n$ be a geometric Markov partition of $f$. 
	
	\begin{enumerate}
		\item The stable boundary of $(f, \mathcal{R})$ is given by $\partial^s \mathcal{R} = \cup_{i=1}^n \partial^s R_i$, and the unstable boundary by $\partial^u \mathcal{R} = \cup_{i=1}^n \partial^u R_i$.
		\item The boundary of $\mathcal{R}$ is the set $\partial \mathcal{R} = \cup_{i=1}^n \partial R_i$.
		\item The interior of the Markov partition $(f, \mathcal{R})$ is the set $\overset{o}{\mathcal{R}} = \cup_{i=1}^n \overset{o}{R_i}$.
	\end{enumerate}
	
	Let $p$ be a periodic point of $f$. In this case, $p$ is a:
	\begin{enumerate}
		\item $s$-\emph{boundary} periodic point of $(f, \mathcal{R})$ if $p \in \partial^s \mathcal{R}$. The set of $s$-boundary points is denoted by $\text{Per}^{s}(f, \mathcal{R})$.
		\item $u$-\emph{boundary} periodic point if $p \in \partial^u \mathcal{R}$. The set of $u$-boundary points is denoted by $\text{Per}^{u}(f, \mathcal{R})$.
		\item \emph{boundary} periodic point if $p \in \partial \mathcal{R}$. The set of boundary periodic points of $f$ is denoted by $\text{Per}^{b}(f, \mathcal{R})$.
		\item \emph{interior} periodic point if $p \in \overset{o}{\mathcal{R}}$. The set of interior periodic points is denoted by $\text{Per}^I(f, \mathcal{R})$.
		\item \emph{corner} periodic point if there exists $i \in \{1, \ldots, n\}$ such that $p$ is a corner point of the rectangle $R_i$. The set of corner periodic points is denoted by $\text{Per}^C(f, \mathcal{R})$.
	\end{enumerate}
\end{defi}

This allows us to introduce distinguished families of Markov partitions.

\begin{defi}\label{Defi: Partion wellsuited/corner/adapted partition}
	Let $\mathcal{R}$ be a Markov partition of the \textbf{p-A} homeomorphism $f$. We say that the Markov partition $\cR$:
	\begin{itemize}
		\item is \emph{well-suited} to $f$ if its only periodic boundary points are singularities, that is, $\textbf{Per}^b(f, \mathcal{R}) = \textbf{Sing}(f)$.
		\item has the \emph{corner property} if every rectangle $R \in \mathcal{R}$ that contains a periodic boundary point $p$ has $p$ as one of its corner points.
		\item is \emph{adapted} to $f$ if it is both well-suited and has the corner property. In other words, its only periodic boundary points are singularities, and each of them is a corner point of any rectangle that contains it.
	\end{itemize}
\end{defi}

\subsection{Geometric types} 

Let $(f, \cR)$ be a pair consisting of a \textbf{p-A} homeomorphism $f$ together with a geometric Markov partition $\cR$ of $f$. The incidence matrix (see \cite{fathi2021thurston}, \cite{bowen1975equilibrium}, \cite{robinson1999dynamical} for more detailed discussions on symbolic dynamics) of the pair is a classic combinatorial object associated with $(f, \cR)$, where the coefficient $a_{ij}$ is equal to the number of horizontal subrectangles of $R_i$ whose image under $f$ is a vertical subrectangle of $R_j$. Toward this association, mathematics has been able to apply results from symbolic dynamics to the study of pseudo-Anosov maps. However, in general, even if two \textbf{p-A} homeomorphisms have Markov partitions with the same incidence matrix, they do not need to be topologically conjugate. This forces us to introduce a new combinatorial object called the geometric type of $(f, \cR)$. 

Let $f$ be a \textbf{p-A} homeomorphism, and let $\cR = \{R_i\}_{i=1}^n$ be a geometric Markov partition. We begin by labeling the horizontal subrectangles of $(f,\cR)$ contained in $R_i$. The elements in $\cH(f,\cR,i)$ are labeled as $\{H^i_j\}_{j=1}^{h_i}$ from the \emph{bottom to the top} with respect to the vertical direction of $R_i$, where $h_i \geq 1$ denotes the number of such horizontal subrectangles.  
Similarly, the vertical subrectangles of $(f, \mathcal{R})$ contained in $R_k$ are labeled from \emph{left to right} with respect to the horizontal direction of $R_k$ as $\{V_l^k\}_{l=1}^{v_k}$, where $v_k \geq 1$ denotes the number of such vertical subrectangles.

\begin{defi}\label{Defi: VH label (f,R)}
	The set of \emph{horizontal} and \emph{vertical} labels of $(f, \mathcal{R})$ are respectively given by 
	\begin{equation*}\label{Equa: H(f,cR)}
		\mathcal{H}(f,\mathcal{R}) = \{(i,j) \mid i \in \{1, \dots, n\} \text{ and } j \in \{1, \dots, h_i\}\}.
	\end{equation*}
	and
	\begin{equation*}\label{Equa: V(f,cR)}
		\mathcal{V}(f,\mathcal{R}) = \{(k,l) \mid k \in \{1, \dots, n\} \text{ and } l \in \{1, \dots, v_k\}\}.
	\end{equation*}
\end{defi}

The homeomorphism $f$ induces a bijection between the set of horizontal and vertical rectangles of $(f, \mathcal{R})$, and thus a bijection between the set of vertical and horizontal labels of $(f, \mathcal{R})$. Consequently, the following equality holds:
$$
\alpha(f,\cR) := \sum_{i=1}^{n} h_i = \sum_{i=1}^n v_i,
$$
and we can introduce the following bijection.

\begin{defi}\label{Defi: Permutation (f,cR)}
	Let $\rho: \cH(f,\cR) \to \cV(f,\cR)$ be the map defined as $\rho(i,j) = (k,l)$ if and only if $f(H^i_j) = V^k_l$.
\end{defi}

The final step is to capture, in another map, the change in the vertical direction induced by $f$ when restricted to one horizontal sub-rectangle of the partition, as defined below.

\begin{defi}\label{Defi: Orientation (f,cR)}
	Suppose $f(H_j^i) = V_l^k$. Let $\epsilon: \mathcal{H}(f, \mathcal{R}) \rightarrow \{1, -1\}$ be the map defined as $\epsilon(i,j) = 1$ if $f$ maps the vertical direction of $H_j^i$ to the vertical direction of $V_l^k$, and $-1$ otherwise.
\end{defi}

The geometric type of $(f,\cR)$ combines all this information.

\begin{defi}[Geometric type]\label{Defi: Geometric type of (f,cR)}
	Let $f: S \to S$ be a \textbf{p-A} homeomorphism, and let $\mathcal{R}$ be a geometric Markov partition of $f$. The \emph{geometric type} of the pair $(f, \mathcal{R})$ is denoted by $T(f, \mathcal{R})$ and is given by:
	\begin{equation}
		T(f, \mathcal{R}) = (n, \{(h_i, v_i)\}_{i=1}^n, \rho, \epsilon).
	\end{equation}
	where:    
	\begin{itemize}
		\item $n$ is the number of rectangles in the Markov partition $\mathcal{R}$.
		\item $h_i$ and $v_i$ are the numbers of horizontal and vertical sub-rectangles of the pair $(f, \mathcal{R})$ contained in the rectangle $R_i \in \mathcal{R}$.
		\item The bijection $\rho$ is as given in Definition \ref{Defi: Permutation (f,cR)}.
		\item The map $\epsilon$ is as given in Definition \ref{Defi: Orientation (f,cR)}.
	\end{itemize}
\end{defi}

\section{An Algorithmic Construction of Markov Partitions}\label{Sec: Contruccion Particion Markov}

In order to construct Markov partitions, the following characterization is particularly useful and was proved in \cite{IntiThesis}

\begin{prop}\label{Prop: Markov criterion boundary}
	Let $f: S \rightarrow S$ be a \textbf{p-A} homeomorphism, and let $\mathcal{R} = \{R_i\}_{i=1}^n$ be a family of rectangles whose union is $S$ and whose interiors are disjoint. Then, $\mathcal{R}$ is a Markov partition for $f$ if and only if the following conditions hold:
	\begin{itemize}
		\item The stable boundary of $\mathcal{R}$, $\partial^s\mathcal{R} := \cup_{i=1}^n \partial^s R_i$, is $f$-invariant.
		\item The unstable boundary of $\mathcal{R}$, $\partial^u\mathcal{R} := \cup_{i=1}^n \partial^u R_i$, is $f^{-1}$-invariant.
	\end{itemize}
\end{prop}

\subsection{Graphs adapted to a \textbf{p-A} Homeomorphism}

The following definition captures the key properties of the stable and unstable boundaries of a Markov partition that we will need in order to construct a graph that is, in fact, the boundary of a Markov partition.

\begin{defi}[Adapted graphs]\label{Defi: Adapted graph}
	Let $\delta^s=\{I_i\}_{i=1}^n$ be a family of compact stable intervals not reduced to a point. The family $\delta^s$ is an $s$-\emph{graph adapted} to $f$ if:
	\begin{itemize}
		\item  The set $\cup \delta^s := \cup_{i=1}^n I_i$ is $f$-invariant,  
		$$f(\cup \delta^s) \subset \cup \delta^s.$$
		\item For all $i\in \{1,\dots,n\}$, there exists $p\in\textbf{Sing}(f)$ such that $I_i \subset \mathcal{F}^s(p)$.
		\item For all $p\in \textbf{Sing}(f)$ and every stable separatrix of $p$, there exists some interval $I_i\in \delta^s$ that is contained in such separatrix and has $p$ as an endpoint.
		\item The endpoints of every interval $I_i\in \delta^s$ are contained in $\mathcal{F}^u(\textbf{Sing}(f))$.
	\end{itemize}
	
	A $u$-\emph{graph adapted} to $f$, $\delta^u := \{J_i\}_{i=1}^n$, is similarly defined, but with the intervals contained in unstable leaves of the singularities, and $\cup \delta^u := \cup_{i=1}^n J_i$ must be $f^{-1}$-invariant. 
\end{defi}

In the following Lemma \ref{Lemm: The generated graph is adapted}, we will construct an adapted $s$-graph starting with a family of $f^{-1}$-invariant unstable arcs.

\begin{lemm}\label{Lemm: The generated graph is adapted}
	Let $\mathcal{J} = \{J_j\}_{j=1}^n$ be a family of nontrivial compact arcs, each contained in $\mathcal{F}^u(\textbf{Sing}(f))$, such that $\cup \mathcal{J} := \cup_{j=1}^n J_j$ is $f^{-1}$-invariant. Let $\delta^s(\mathcal{J}) = \{I_i\}_{i=1}^m$ be the family of all stable segments that have one endpoint in \textbf{Sing}$(f)$ and the other in $\cup \mathcal{J}$, with their interiors disjoint from $\cup \mathcal{J}$.
	
	In this situation, the family $\delta^s(\mathcal{J})$ is an $s$-graph adapted to $f$, and we call $\delta^s(\mathcal{J})$ the \emph{$s$-graph adapted to $f$ generated by} $\mathcal{J}$.
\end{lemm}

\begin{proof}
	By construction, each interval in the graph $\delta^s(\mathcal{J})$ is contained in $\mathcal{F}^s(\textbf{Sing}(f))$, has one endpoint in $\textbf{Sing}(f)$ and the other in $\mathcal{J} \subset \mathcal{F}^u(\textbf{Sing}(f))$. Moreover, for every stable separatrix of any point in $\textbf{Sing}(f)$, there exists a segment in $\delta^s(\mathcal{J})$ that belongs to that separatrix. It remains to show that $\cup \delta^s(\mathcal{J})$ is $f$-invariant.
	
	Let $I := I_i$ be an interval in $\delta^s(\mathcal{J})$. Since $I$ is contained in a stable separatrix of a point $p \in \textbf{Sing}(f)$, we know that $f(I)$ is contained in a stable separatrix of $f(p)$, and within that separatrix, there exists an interval $I'$ of $\delta^s(\mathcal{J})$. If $f(I)$ is not a subset of $I'$, then $I' \subset f(I)$, which implies that $I'$ has an endpoint $x'$ in $\cup \mathcal{J}$ that belongs to the interior of $f(I)$. Since $\cup \mathcal{J}$ is $f^{-1}$-invariant, we have $f^{-1}(x') \in \cup \mathcal{J}$, and this point is in the interior of $I$, since the interior of $I$ cannot intersect $\cup \mathcal{J}$, this is a contradiction. Consequently, $f(I) \subset I' \subset \cup \mathcal{J}$, and thus $\cup \delta^s(\mathcal{J})$ is $f$-invariant.
\end{proof}

If $\mathcal{I}$ is a family of stable segments contained in $\mathcal{F}^s(\textbf{Sing}(f))$ and its union is $f$-invariant, the $u$-\emph{graph adapted to $f$ generated by} $\mathcal{I}$ is similarly defined and is denoted by $\delta^u(\mathcal{I})$.

\subsection{Compatible Graphs} Even if $\delta^s$ and $\delta^u$ are graphs adapted to $f$, their union is not necessarily the boundary of a Markov partition of $f$. We must require a more subtle relation between the adapted graphs in order for them to bound bi-foliated open discs and then take the rectangles in our partition as the closures of such discs. The relation between them is given in \ref{Defi: Compatible graphs}. Before presenting it, we need to introduce a few new concepts.

\begin{defi}\label{Defi: Regular part delta-s}
	Let $\delta^s=\{I_i\}_{i=1}^n$ be an $s$-graph adapted to $f$. The \emph{regular part} of $\delta^s$ is  $\overset{o}{\delta^s}:=\cup \delta^s \setminus \cup_{i=1}^n \{ \partial I_i \}$. The regular part of an adapted $u$-graph is similarly obtained by removing the endpoints of each interval in $\delta^u$ and taking their union.
\end{defi}

\begin{defi}\label{Defi: Rail regular/extremal}
	Let $\delta^s$ be an $s$-graph adapted to $f$.
	\begin{itemize}
		\item  A \emph{regular $u$-rail} of $\delta^s$ is a unstable interval $J$ whose interior is disjoint from $\cup \delta^s$ but has endpoints in $\overset{o}{\delta^s}$.
		
		\item An \emph{extreme $u$-rail} of $\delta^s$ is an unstable segment $J$ that has at least one endpoint not contained in $\overset{o}{\delta^s}$, and for which there exists an \emph{embedded rectangle}\footnote{This means the parametrization of $R$ is a homeomorphism in the unit square $\II^2$} $R$ such that: $\partial^s R \subset \cup \delta^s$, with $J$ being a $u$-boundary component of $R$, and such that any other vertical leaf of $R \setminus J$ is a regular $u$-rail of $\delta^s$.
		
	\end{itemize}
	
	The set of extreme $u$-rails of $\delta^s$ is denoted by \textbf{Ex}$\,^u(\delta^s)$.
\end{defi}

\begin{defi}\label{Defi: Compatible graphs}
	Let $\delta^s = \{I_i\}_{i=1}^n$ and $\delta^u = \{J_j\}_{j=1}^m$ be $s$-graphs and $u$-graphs adapted to $f$, respectively. They are \emph{compatible} if they satisfy the following properties:
	\begin{itemize}
		\item The endpoints of $\delta^s$ belong to $\cup \delta^u := \cup_{j=1}^m J_j$, and the endpoints of $\delta^u$ belong to $\cup \delta^s := \cup_{i=1}^n I_i$.
		\item The extreme $u$-rails of $\delta^s$ are contained in $\cup \delta^u$, and the extreme $s$-rails of $\delta^u$ are contained in $\cup \delta^s$.
	\end{itemize}
\end{defi}

In view of Lemma \ref{Lemm: The generated graph is adapted}, we know how to construct $s$- or $u$-graphs adapted to $f$. The next step is to determine how we can build a pair of adapted graphs.

\begin{lemm}\label{Lemm: Iteration to be adapted}
	Let $\delta^s$ and $\delta^u$ be graphs adapted to $f$. Then there exists an $n := n(\delta^s, \delta^u) \in \mathbb{N}$ such that $\delta^s$ and $h^n(\delta^u)$ are compatible whenever $m \geq n$, and $n$ is the minimum natural number with this property.
\end{lemm}

\begin{proof}
	
	First, we prove that there exists $N_1 \in \mathbb{N}$ such that for all $n > N_1$, $\delta^s$ contains the extreme points of $f^n(\delta^u)$ while $f^n(\delta^u)$ contains the extreme points of $\delta^s$ too. Next, we must obtain certain $N_2 \in \mathbb{N}$ such that for all $n > N_2$, the extreme $u$-rails of $\delta^s$ are in $h^n(\delta^u)$ and the extreme $s$-rails of $h^n(\delta^u)$ are contained in $\delta^s$. These conditions imply that:
	
	$$
	\{n \in \mathbb{N} \mid \delta^s \text{ is compatible with } f^n(\delta^u)\} \neq \emptyset,
	$$
	
	so we can define $n := n(\delta^s, \delta^u)$ as the minimum of this set.
	
	Let's start by assuming that $\delta^u = \{J_j\}_{j=1}^h$ and $\delta^s = \{I_i\}_{i=1}^l$. We define:
	$$
	L^u := \max\{\mu^s(J_j) \mid 1 \leq j \leq h\}.
	$$
	It is clear that for every $n \in \mathbb{N}$ and every interval $J_j \in \delta^u$:
	$$
	\mu^s(f^n(J_j)) = \lambda^n \mu^s(J_j) \leq \lambda^n L^u.
	$$
	
	Let $z \in \delta^s$ be an extreme point of $\delta^s$. Since $\delta^s$ is an adapted graph, the point $z$ lies on the unstable leaf of a singularity of $f$. Let $[p_z, z]^u$ denote the unique unstable arc  passing through $z$ and connecting it with a singularity $p_z$ of $\cF^u$. Let's define:
	$$
	F^u := \max\{\mu^s([p_z, z]^u) \mid z \text{ is an extreme point of } \delta^s\}.
	$$
	
	By the uniform expansion and density of every unstable leaves of $\mathcal{F}^u$ in $S$, there exists $n_1 \in \mathbb{N}$ such that for every $n \geq n_1$,  $\lambda^n L^u > F^u$. Moreover, if $n > n_1$ and $z$ is an extreme point of $\delta^s$, there exists a unique $J_j$ in $\delta^u$ such that $z$ lies on the same separatrix as $f^n(J_j)$ and then:
	$$
	\mu^u(f^n(J_j)) \geq \lambda^n L^u > F^u \geq \mu^s([p_z, z]^u),
	$$
	implying that $z$ belongs to $f^n(J_j)$, or equivalently, $z$ belongs to $f^n(\delta^u)$ for all $n \geq n_1$.
	
	In the same way, but using $f^{-1}$ and the measure $\mu^u$, we can deduce the existence of $n_2 \in \mathbb{N}$ such that for all $n \geq n_2$ and every extreme point $z$ of $\delta^u$, $z$ is also an extreme point of $f^{-n}(\delta^s)$. In other words, for all $n \geq n_2$ and every extreme point $z$ of $f^n(\delta^u)$, $z$ is contained in $\delta^s$. Let's take 
	$$N_1 = \max\{n_1, n_2\}$$.
	
	Consider the set of $u$-extreme rails of $\delta^s$.  For each $p \in \textbf{Sing}(f)$ and each unstable separatrix $F^u_i(p)$ of $p$ (if $p$ is a $k$-prong, we consider $i = 1, \dots, k$), we define $J(p, i) \subset F^u_i(p)$ as the minimal compact interval containing the intersection of all the $u$-extreme rails of $\delta^s$ with the separatrix $F^u_i(p)$. In case some $u$-extreme rail $J$ of $\delta^s$ contains a singularity in its interior, we consider only the subinterval of $J$ contained in the separatrix $F^u_i(p)$. Since the  set of $u$-extreme rails of $\delta^s$ is and each $u$-extreme rail is a closed interval the following quantity  is finite:
	
	$$
	M^u := \max\{ \mu^s(J(p, i)) \mid p \in \textbf{Sing}(f) \text{ is a $k$-prong and } i = 1, \dots, k \}.
	$$
	
	Like $\delta^u = \{J_j\}_{j=1}^h$, the next quantity is finite too:
	$$
	G^u = \min\{\mu^s(J_j) \mid j = 1, \dots, h\}.
	$$
	
	By the uniform expansion in unstable leaves of $\mathcal{F}^u$, there exists $n_1 \in \mathbb{N}$ such that for all $n \geq n_1$, $\lambda^n G^u \geq M^u$. Furthermore, for any $k$-prong $p \in \textbf{Sing}(f)$ and every $i = 1, \dots, k$, there exists an interval $J_j$ in $\delta^u$ such that $f^n(J_j)$ is contained in the same unstable separatrix as $J(p, i)$. Then we have the following computation:
	
	$$
	\mu^s(f^n(J_j)) = \lambda^n \mu^s(J_j) \geq \lambda^n G^u \geq M^u.
	$$
	
	Since $\mu^s(J(p, i)) \leq M^u$, the inequality implies that $J(p, i) \subset f^n(J_j)$. By the construction of $J(p, i)$, we deduce that for each $n \geq n_1$, any $u$-extreme rail of $\delta^s$ is contained in $f^n(\delta^u)$.
	
	A similar proof using $f^{-1}$ and the measure $\mu^u$ gives another natural number $n_2 \in \mathbb{N}$ such that for all $n \geq n_2$, any extreme $s$-rail of $\delta^u$ is contained in $f^{-n}(\delta^s)$, or equivalently, all extreme $s$-rails of $f^n(\delta^u)$ are contained in $\delta^s$. Let $N_2 = \max(n_1, n_2)$, to finally take 
	$$
	N:=\max\{N_1,N_2\}.$$
	Therefore the following quantity exists:
	$$
	n(\delta^s, \delta^u) := \min \{n \in \mathbb{N} \mid \forall n \geq N, \, \delta^s \text{ is compatible with } f^n(\delta^u) \},
	$$
	
	and this is the number that we were looking for.
\end{proof}

\subsection{The existence of Markov partitions}Our previous construction of adapted and compatible graphs, along with the next proposition, implies the existence of Markov partitions.

\begin{prop}\label{Prop:Compatibles implies Markov partition}
	Let $\delta^u$ and $\delta^s$ be graphs adapted to $f$ and compatible. Then the closure in $S$ of each connected component of
	\[
	\overset{o}{S} := S \setminus \left( \bigcup \textbf{Ex}^s(\delta^u) \,\cup\, \bigcup \textbf{Ex}^u(\delta^s) \right)
	\]
	is a rectangle whose stable (resp.\ unstable) boundary is contained in $\bigcup \delta^s$ (resp.\ $\bigcup \delta^u$). Moreover, the family $\mathcal{R}(\delta^s,\delta^u)$ of these rectangles  is an adapted Markov partition for $f$.
\end{prop}

The proof will be arrive after two consecutive lemmas.

\begin{lemm}\label{Lemm: Finite c.c.}	
	The set $\overset{o}{S}=S\setminus (\delta^s \cup \text{Ex}^u(\delta^s))$ has a finite number of connected components.
\end{lemm}  

\begin{proof}
	Since there are only a finite number of endpoints in $\delta^s$ and $\delta^u$, and a finite number of stable and unstable separatrices passing through each endpoint, there is a finite number of (compact) extreme rails of $\delta^s$ and $\delta^u$. Finally, since $\delta^s \cup \text{Ex}^u(\delta^s)$ consists of a finite number of compact intervals, its complement must have a finite number of connected components.
\end{proof}

\begin{lemm}\label{Lemm: Equivalent class rectangle}
	Let $r$ be a connected component of $\overset{o}{S}$. Then $R:=\overline{r}$ is a rectangle adapted to $f$, whose stable and unstable boundaries are contained in $\cup\delta^s$ and $\cup\delta^u$, respectively.
\end{lemm}

\begin{proof}
	It is clear that $r$ is open, connected, its closure is compact, and it does not contain any singularity in its interior, so, for all $x\in r$ the connect components of  $r\cap F^s_x$ and  $r\cap F^u_x$  are homeomorphism to opens intervals.
	For every $x\in r$ there are vertical rectangular open neighborhood $V_x$ as in figure \ref{Fig: Equiva Rectangles}, and all of them  covers $r$. 
	
	\begin{figure}[ht]
		\centering
		\includegraphics[width=0.5\textwidth]{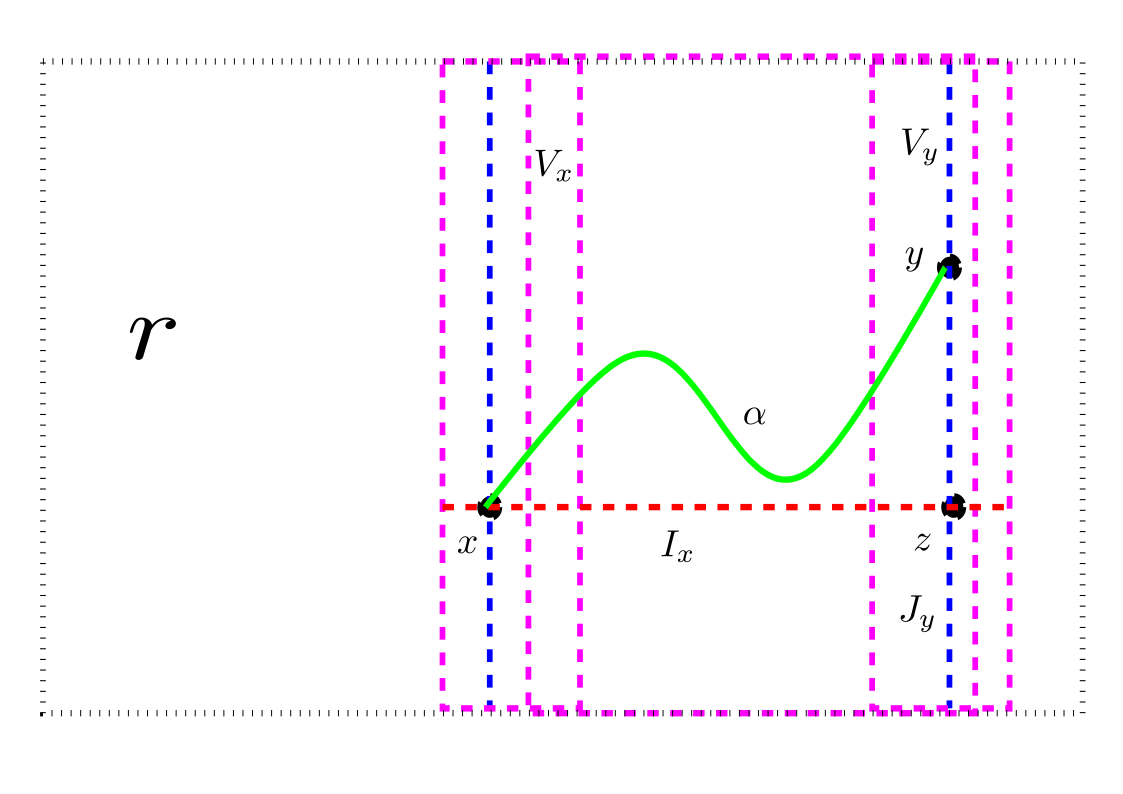}
		\caption{ Vertical rectangular neighborhoods }
		\label{Fig: Equiva Rectangles}
	\end{figure}

	Take two point $x,y\in r$ and draw a curve $\alpha \subset r$ that joint. Since $\alpha$ is compact there is a finite number of rectangular neighborhoods $H_1, H_2,\cdots, H_n$  that covers $\alpha$ and in particular it union contains $x$ and $y$. Clearly in the union $\cup_{i=1}^n H_i$ the arcs $I_x$ and $J_y$ intersect just one time. But  $I_x$ and $J_y$  have more than one intersection lets to said $a$ and $b$, we repite our game and take a curve that $\beta\subset r$ that joint $a,b,x$ and$y$. This will produce a cover iof the curve by rectangular neighborhoods inse that thatre is contained $a$ and $b$ but in this new family the segments $I_x$ and $J_y$ must intersect one time, so $a$ mut tyo be equal to $ b$. Therefore $I_x$ and $J_y$  intersect in a single point isede $r$.
	
	By Definition \ref{Defi: Rectangulo}, $R := \overline{r}$ is a rectangle. Furthermore, its stable and unstable boundaries must lie in the adapted $s$ and $u$ graphs, as they must be contained in $S \setminus \overset{o}{S}$.
\end{proof}
Now we are ready to prove Proposition \ref{Prop:Compatibles implies Markov partition}.

\begin{proof}
	In view of Lemma \ref{Lemm: Finite c.c.} and Lemma \ref{Lemm: Equivalent class rectangle}, the family $\cR$ is a finite collection of rectangles, as their interiors correspond to different connected components of $\overset{o}{S}$, and their interiors are disjoint. At the same time, since $\partial^{s,u} \cR = \delta^{s,u}$, the stable boundary of $\cR$ is $f$-invariant, and its unstable boundary is $f^{-1}$-invariant. By Proposition \ref{Prop: Markov criterion boundary}, $\cR$ is a Markov partition of $f$.
	The fact that $\delta^s$ and $\delta^u$ are adapted and compatible implies directly that $\cR$ is an adapted Markov partition.
\end{proof}

Now, we will summarize the steps required to construct a Markov partition, all of which have been previously discussed.

\begin{cons}\label{Cons: Recipe for Markov partitions}
	Let $p$ and $q$ be singular points of $f$, with separatrices $F^s(p)$ and $F^u(q)$ respectively, and let $z$ be a point in their intersection $F^s(p)\cap F^u(q)$. Consider the following construction:
	\begin{enumerate}
		\item Define $J^u(z)$ as the unstable interval $[p,z]^u$ in $F^u(p)$. We call it \emph{primitive segment}.
		
		\item Define $\cJ^u(z)=\cup_{i\in \NN}f^{-i}(J^u(z))$. Due to contraction in the unstable foliation, $\mathcal{J}^u(z)$ is a finite union of closed intervals and is $f^{-1}$-invariant.
		
		\item The graph $\mathcal{J}^u(z)$ satisfies the conditions of Lemma \ref{Lemm: The generated graph is adapted} and thus its generated $s$-graph is adapted to $f$. We denote it as $\delta^s(z)$.
		
		\item The graph $\delta^s(z)$ is $f$-invariant, and by Lemma \ref{Lemm: The generated graph is adapted} (applied to $f^{-1}$), its generated $u$-graph, $\delta^u(z)$, is adapted to $f$.
		
		\item By Lemma \ref{Lemm: Iteration to be adapted}, there exists a number $n(z) = n(\delta^s(z), \delta^u(z))$, called the \emph{compatibility coefficient} of $z$, such that, for all $n > n(z)$, $\delta^s(z)$ and $f^n(\delta^u(z))$ are compatible.
		
		\item Finally, for all $n\geq n(z)$, Proposition \ref{Prop:Compatibles implies Markov partition} implies the existence of an adapted Markov partition $\mathcal{R}(z,n)$ with a stable boundary equal to $\cup \delta^s(z)$ and an unstable boundary equal to $\cup f^n(\delta^u(z))$.
	\end{enumerate}
\end{cons}

Therefore, we recover the classical result of the existence of a Markov partition, along with a bit more: the only periodic points of our Markov partition are singularities, and each singularity is located at a corner of the rectangle that contains it.

\begin{coro}\label{Coro: Existence adapted Markov partitions}
	Every generalized \textbf{p-A} homeomorphism has adapted Markov partitions.
\end{coro}

\section{Primitive Markov partitions and their geometric types}\label{Sec: Particiones primitivas}

There exists an infinite number of Markov partitions, whether adapted or not, for a \textbf{p-A} homeomorphism $f$. It is important, however, to identify a distinguished class among them.   In this section, we will apply Construction \ref{Cons: Recipe for Markov partitions} to a family of distinguished points known as \emph{first intersection points}.

This approach allows us to define a natural number $n(f) \in \mathbb{N}$, which depends only on the conjugacy class of our homeomorphism $f$. For all $n \geq n(f)$, we can then construct a family $\mathcal{M}(f, n)$ of geometric Markov partitions, called primitive geometric Markov partitions. The importance of these Markov partitions lies in the fact that the set of geometric types of the geometric Markov partitions in $\mathcal{M}(f, n)$ is finite.

This class of problems has been approached by I. Agol \cite{agol2011ideal}, who was able to construct a distinguished family of ideal triangulations for mapping tori arising from pseudo-Anosov homeomorphisms. We wish to explore the relationships between Agol's construction and our primitive Markov partitions, which will be addressed in a future article.

\subsection{First intersection points}

\begin{defi}\label{Defi: first intersection points}	
	Let $f$ be a \textbf{p-A} homeomorphism, $p, q \in \textbf{Sing}(f)$, $F^s(p)$ be a stable separatrix of $p$, and $F^u(q)$ be an unstable separatrix of $q$.
	A point $x \in [\mathcal{F}^s(p) \cap \mathcal{F}^u(q)] \setminus \textbf{Sing}(f)$ is a \emph{first intersection point} of $f$ if the stable interval $[p,x]^s \subset F^s(p)$ and the unstable segment $[q,x]^u \subset F^u(q)$ have disjoint interiors. In other words:
	$$
	(p,x]^s\cap (q,x]^u=\{x\}.
	$$
\end{defi}

\begin{lemm}\label{Lemm: Existencia puntos de primera intersecion}
	There exists at least one first intersection point for $f$
\end{lemm}

\begin{proof}
	Let $I$ be a compact interval contained in $F^s(p)$ with one endpoint at $p$. Consider a singularity $q$ (which could be equal to $p$) of $f$. Since any unstable separatrix $F^u(q)$ of $q$ is dense in the surface, there exists a closed interval $J \subset F^u(q)$ such that $\overset{o}{J} \cap \overset{o}{I} \neq \emptyset$.  
	
	Since $J$ and $I$ are compact sets, their intersection $J \cap I$ consists of a finite number of points, $\{z_0, \dots, z_n\}$. We can assume that $n > 1$, and we orient the interval $J$ pointing towards $q$. With this orientation, we have $z_i < z_{i+1}$. In this manner:
	\begin{itemize}
		\item If $p = q$, then $z_0 = p = q$, but $p \neq z_1$, and we take $z = z_1$.
		\item If $p \neq q$, then $p \neq z_0 \neq q$, and we take $z = z_0$.
	\end{itemize}
	
	Clearly, $(p,z]^s \cap (q,z]^u = \{z\}$, and $z$ is a first intersection point.
\end{proof}

\begin{lemm} \label{Lemm: Image of first intersection is first intersection}
	If $z$ is a first intersection point of $f$, then $f(z)$ is also a first intersection point of $f$.
\end{lemm}

\begin{proof}
	If $z$ is a first intersection point, then there exist singular points $p$ and $q$ of $f$ such that:
	$$
	\{z\} = (p,z]^s \cap (q,z]^u.
	$$
	Therefore, applying $f$, we obtain:
	$$
	\{f(z)\} = (f(p),f(z)]^s \cap (f(q),f(z)]^u,
	$$
	which implies that $f(z)$ is also a first intersection point of $f$.
\end{proof}

The following proposition justifies our attention on first intersection points.

\begin{prop}\label{Prop: Finite number of first intersection points}
	Let $f$ be a \textbf{p-A} homeomorphism. Then $f$ has a finite number of orbits of first intersection points.
\end{prop}

Since there is only a finite number of singularities and separatrices of $f$, Proposition \ref{Prop: Finite number of first intersection points} follows directly from the next lemma.

\begin{lemm}\label{lemm: finite intersect in fundamental domain}
	Let $p$ and $q$ be singularities of $f$, and let $F^s(p)$ and $F^u(q)$ be stable and unstable separatrices of these points. There exists a finite number of orbits of first intersection points contained in $F^s(p)\cap F^u(q)$
\end{lemm}

\begin{proof}
	Let $n\in \NN$ be the smallest natural number such that $g:=f^n$ fixes the separatrices $F^s(p)$ and $F^u(q)$. 	Take a point $z_0 \in F^s(p)\cap F^u(q)$, which may or may not be a first intersection point of $f$. The interval $(g(z_0),z_0]^s \subset F^s(p)$ is a fundamental domain for $g$ and it contains a fundamental domain for $f$. Therefore we need to show that within this interval $(g(z_0),z_0]^s$ there exists a finite number of first intersection points.
	
	Consider a first intersection point $z$ in $F^s(p) \cap F^u(q)$. There exists an integer $k \in \mathbb{Z}$ such that $g^k(z)\in (g(z_0),z_0]^s \cap F^u(q)$. Hence there are two possible configurations:
	
	\begin{enumerate}
		\item Either $g^k(z)\in (g(z_0),z_0]^s \cap (q,g(z_0)]^u$, or
		\item $g^k(z)\in (g(z_0),z_0]^s \cap (g(z_0),\infty)^u$.
	\end{enumerate}
	
	We claim that option $(2)$ is not possible. If such a configuration existed, it would imply that $g(z_0) \in (p,g^k(z)]^s \cap [q,g^k(z)]^u$, meaning that $g^k(z)$ would not be a first intersection point. However, Lemma \ref{Lemm: Image of first intersection is first intersection} implies that $f^{kn}(z)$ is always a first intersection point, leading to a contradiction.
	
	Hence, we conclude that $g^k(z) \in (g(z_0),z_0]^s \cap [q,g(z_0)]^u \subset [g(z_0),z_0]^s \cap [q,g(z_0)]^u$, which is the intersection of two compact intervals. Therefore, this intersection is a finite set.
	
	This implies that every first intersection point $z\in F^s(p) \cap F^u(q)$ has some iteration lying in a finite set. Consequently, there are only a finite number of orbits of first intersection points in $F^s(p)\cap F^u(q)$.
\end{proof}

\subsection{Primitive geometric Markov partitions} Before continuing, let us recall that two homeomorphisms $f: S_f \to S_f$ and $g: S_g \to S_g$ are topologically conjugate if there exists a homeomorphism $h: S_f \to S_g$ such that $f = h^{-1} \circ g \circ h$.

\begin{defi}\label{Defi: Primitive Markov partitions}
	Let $z$ be a first intersection point of $f$, and let $n(z)$ be the compatibility coefficient of $z$. Then, for all $n \geq n(z)$, a  Markov partition $\mathcal{R}(z,n)$ constructed as indicated in \ref{Cons: Recipe for Markov partitions} is called  \emph{primitive Markov partition} of $f$ generated by $z$ of order $n$.
\end{defi}

\begin{prop}\label{Prop: Conjugates then primitive Markov partition}
	Let $f: S_f \rightarrow S_f$ and $g: S_g \rightarrow S_g$ be two \textbf{p-A} homeomorphisms that are topologically conjugate via a homeomorphism $h: S_f \rightarrow S_g$, i.e.\ $g = h \circ f \circ h^{-1}$. Let $z \in S_f$ be a first intersection point of $f$. Then:
	\begin{itemize}
		\item[i)] $h(z)$ is a first intersection point of $g$.
		\item[ii)] $h(\cJ^u(z)) = \cJ^u(h(z))$.
		\item[iii)] $h(\delta^s(z)) = \delta^s(h(z))$ and $h(\delta^u(z)) = \delta^u(h(z))$.
		\item[iv)] The graph $\delta^s(z)$ is compatible with $f^n(\delta^u(z))$ if and only if $\delta^s(h(z))$ is compatible with $g^n(\delta^u(h(z)))$.
		\item[v)] The compatibility coefficients of $z$ and $h(z)$ coincide: $n(z)=n(h(z))$.
		\item[vi)] A class $r$ of $u$-rails for $\delta^s(z)$ corresponds under $h$ to a class of $u$-rails for $\delta^s(h(z))$, and likewise for classes of $s$-rails of $f^n(\delta^u(z))$ and $g^n(\delta^u(h(z)))$.
		\item[vii)] For every $n \ge n(z)=n(h(z))$, the image $h(\cR(z,n)) = \cR(h(z),n)$ is a primitive Markov partition of $g$ generated by $h(z)$ and of order $n$.
	\end{itemize}
\end{prop}

\begin{proof}
	First, $h$ send transverse foliations of $f$ into transverse foliations of $g$ and singularities of $f$ into singularities of $g$, that is how we use the hypothesis that $f$ and $g$ are pseudo-Anosov homeomorphisms. But even more, $h$ sends disjoint segments of stable and unstable leaves of $f$ to disjoint segments of stable and unstable leaves of $g$, therefore, if $z$ is first intersection point of $f$, $h(z)$ is first intersection point of $g$. This shows the accuracy of Item $i)$.	A direct computation proof Item $ii)$:
	$$
	h(\cJ^u(z))=h(\cup f^{-n}(J^u(z)))=\cup g^{-n}(J^u(h(z)))= \cJ^u(h(z)).
	$$

	A compact interval $[p,x]^s \in \delta^s(z)$  have its interior disjoint from $\cJ^u(z)$ and  endpoint  $x\in \cJ^u(z)$. Then $h([p,x]^s)$ is a segment with disjoint interior of $h(\cJ^u(z))=\cJ^u(h(z))$ but endpoint $h(x)\in \cJ^u(h(z))$, this implies $h([p,x]^s)$ is contained in $\delta^s(h(z))$, hence $h(\delta^s(z))\subset \delta^s(h(z))$. The symmetric argument starting from the segment  $[p',x']^s\in \delta^s(h(z))$ and using $h^{-1}$ gives the equality. Analogously it is possible to show that $\delta^u(h(z))=h(\delta^u(z))$. The arguments in this paragraph proved the Item $iii)$.	
	
	Clearly the point $e$ is an extreme point of $\delta^s(z)$ if and only if $h(e)$ is an extreme point of $h(\delta^s(z))=\delta^s(h(z))$, moreover $e\in f^n(\delta^n(z))$ if and only if $h(e)\in g^{n}(\delta^u(h(z)))$. Therefore an extreme point $e$ of $\delta^s(z)$ is contained in $f^n(\delta^n(z))$ if and only if the extreme point $h(e)$ of $\delta^s(h(z))$ is contained in $g^{n}(\delta^u(h(z))$.	Another consequence is that the regular part of $\delta(z)$ is sent by $h$ to the regular part of $\delta(z)$ and an $u$-regular $J$-rail of $\delta(z)$ is such that $h(J)$ is a regular rail of $\delta^s(h(z))$.
	
	We claim that $J$ is an extremal rail of $\delta^s(z)$ if and only if $h(J)$ is an extremal lane of $\delta^u(h(z))$. Indeed if $R$ is the embedded rectangle of $S_f$ given by the definition of extreme rail having $J$ as a vertical boundary  component, therefore the embedded rectangle $h(R)$ satisfies the definition for $h(J)$ to be an extreme rail of $\delta^s(h(z))$, since  except for $h(J)$ all its other vertical segments are regular rails of $\delta^s(h(z))$ and it stable boundary is contained in $h(\delta^s(z))=\delta^s(h(z))$. Moreover, $J \subset f^n(\delta^u(z))$ if and only if $h(J)\subset g^n(\delta^u(h(z)))$.
	
	This proves that $\delta^s(z)$ is compatible with  $f^n(\delta^u(z))$ if and only if $\delta^s(h(z))$ is compatible with $g^n(\delta^u(h(z)))$. In this way the point $iv)$ is corroborated, but at the same time we deduce that $n(z)=n(h(z))$ because they are the minimum of the same set of natural numbers and thus the point $v)$ is obtained.
	
	Let $r$ be an equivalent class of $u$-rails for $\delta^s(z)$, as stated before $I$ is a regular rail for $\delta^s(z)$ if and only if $h(I)$ is a regular rail for $\delta^s(h(z))$. 	We denote $I\equiv_{\delta^s(z)}I'$ to indicate that any point in $I$ is $\delta^s(z)$ equivalent to any point in $I'$. In this manner $h(I)\equiv_{\delta^s(h(z))}h(I')$, because the image by $h$ of rectangles and regular rails realizing the equivalence between points in $I$ and points in $I'$ are rectangles and regular rails realizing the equivalence between points in $h(I)$ and points in ´$h(I')$. This implies that $h(r)$ is an equivalent class of $u$-rails for $\delta^s(h(z))$. Analogously the image by $h$ of an equivalent class $r'$ of $s$-rails for $f^n(\delta^u(z)$ is an equivalent class of $s$-rails for $g^n(\delta^u(h(z)))$.

	As the interior of a rectangle $R$ of $R(z,p)$ is a connected component of the intersection a equivalent class $r$ of $u$-rails of $\delta^s(z)$ with  a class $r'$ of $s$-rails for $f^n(\delta^u(z))$, then $h(R)$ is a connected component of the intersection of $h(r)\cap h(r')$ and $h(r)$ and $h(r')$ are $s,u$-rail classes for $\delta^s(h(z))$ and $g^n(\delta^u(h(z)))$, by definition $h(\overset{o}{R})$ correspond to the interior of a rectangle of $\cR(h(z),n)$ and then $h(\cR(z,n))=\cR(h(z),n)$. In this manner we obtain Item $vi)$

	Since the interior of a rectangle  $R$ of $R(z,p)$ is a connected component of the intersection an $r$ equivalent class of $u$-rails of $\delta^s(z)$ with an $r'$ class of $s$-rails for $f^n(\delta^u(z))$, then $h(R)$ is a connected component of the intersection of $h(r)\cap h(r')$, anyway $h(r)$ and $h(r')$ are $s,u$-rail classes for $\delta^s(h(z))$ and $g^n(\delta^u(h(z)))$, therefore $h(\overset{o}{R})$ correspond to the interior of a rectangle of $\cR(h(z),n)$ and then $h(\cR(z,p))=\cR(h(z),p)$. This probe the Item $vii)$ and finish our proof.
\end{proof}

When Proposition \ref{Prop: Conjugates then primitive Markov partition} is applied to $f = g = h$, we obtain that $f = f \circ f \circ f^{-1}$. It then follows that $n(z) = n(f(z))$ for every first intersection point $z$ of $f$. Therefore, the quantity $n(f^m(z))$ remains constant over the entire orbit of $z$ and we have the following corollary.

\begin{coro}\label{Coro: n(f) number}
	There exists $n(f) \in \NN$, called the \emph{compatibility order} of $f$, such that:
	$$
	n(f) = \max\{n(z) : z \text{ is a first intersection point of } f\}.
	$$
\end{coro}

In fact, the compatibility order of $f$ is invariant under topological conjugacy, as stated in the next result.

\begin{coro}\label{Coro: n(f) conjugacy invariant}
	Let $f: S_f \rightarrow S_f$ and $g: S_g \rightarrow S_g$ be two \textbf{p-A} homeomorphisms that are conjugated through a homeomorphism $h: S_f \rightarrow S_g$. Then the compatibility order of $f$ and $g$ coincides, i.e., $n(f) = n(g)$.
\end{coro}

\begin{proof}
	Items $i)$ and $v)$ of proposition \ref{Prop: Conjugates then primitive Markov partition} imply the following set equalities:
	\begin{align*}
		\{n(z): z \text{ is a first intersection point of } f \}= \\
		\{n(h(z)): z \text{ is a first intersection point of } f \}=\\
		\{n(z'): z' \text{ is a first intersection point of } g \}.
	\end{align*}
	Therefore, it follows that its maximum is the same and then $n(f) = n(g)$.
\end{proof}

We are ready to  introduce a distinguished set of Markov partitions of $f$.

\begin{defi}\label{Defi: primitive Markov parition}
	Let $n \geq n(f)$. The set of primitive Markov partitions of $f$ of order $n$ consists of all the Markov partitions of order $n$ generated by the first intersection points of $f$, and is denoted by:
	$$
	\mathcal{M}(f,n) := \{\mathcal{R}(z,n) : z \text{ is a first intersection point of } f\}.
	$$
\end{defi}

Another application of item $vii)$ in Proposition \ref{Prop: Conjugates then primitive Markov partition} for the case when $f=g=h$ is that if $z$ is a first intersection point of $f$ and $n\geq n(f)$, then $f(\cR(z,n))=\cR(f(z),n)$. This observation yields the following corollary:

\begin{coro}\label{Coro: Finite orbits of primitive Markov partitions}		
	Let $n\geq n(f)$. Then, there exists a finite but nonempty set of orbits of primitive Markov partitions of order $n$.
\end{coro}

\begin{proof}
	As $n\geq n(f)$, there exists at least one Markov partition $\cR(z,n)$ in $\cM(f,n)$, and therefore the orbit of $\cR(z,n)$, given by $\{\cR(f^m(z),n)\}_{m\in \ZZ}$, is contained in $\cM(f,n)$.
	
	Let $\{z_1, \cdots, z_k\}$ be a set of first intersection points of $f$ such that any other first intersection point can be written as $f^m(z_i)$ for a unique $i \in \{1,\cdots, k\}$, and no two different points $z_i$ and $z_j$ are in the same orbit. . Let $\cR(z,n)$ be a primitive Markov partition of order $n$. If $z$ is any first intersection point, $z=f^m(z_i)$ and we have:
	$$
	f^m(\cR(z_i,n))=\cR(f^m(z_i),n)=\cR(z,n). 
	$$
	Therefore, in $\cM(f,n)$, there are at most $k$ different orbits of primitive Markov partitions of order $n$.
\end{proof}

\subsection{Primitive geometric types}We are interested in studying the set of geometric types produced by the orbit of a primitive Markov partition. Therefore, it is important to understand the behavior of a geometric Markov partition under the action of a homeomorphism that preserves the orientation.

\begin{defi}\label{Defi: induced geometric Markov partition}
	Let $f: S \rightarrow S$ be a \textbf{p-A} homeomorphism, let $h: S \rightarrow S$ be an orientation-preserving homeomorphism, and let $g := h \circ f \circ h^{-1}$.
	
	If $\mathcal{R}$ is a geometric Markov partition of $f$, the geometric Markov partition of $g$ \emph{induced} by $h$ is the Markov partition $h(\mathcal{R}) = \{ h(R_i) \}_{i=1}^n$. To each rectangle $h(R_i)$, we choose the unique orientation in the vertical and horizontal foliations such that $h$ preserves both orientations at the same time.
\end{defi}

The following lemma clarifies the correspondence between the horizontal and vertical rectangles of the partition $\mathcal{R}$ and those of $h(\mathcal{R})$.

\begin{lemm}\label{Lemm: Conjugated partitions same type.}
	Let $f$ and $g$ be two generalized pseudo-Anosov homeomorphisms conjugated through a homeomorphism $h$ that preserves the orientation. Let $\cR=\{R_i\}_{i=1}^n$  be a geometric Markov partition for $f$, and let $h(\cR)=\{h(R_i)\}_{i=1}^n$ be the geometric Markov partition of $g$ induced by $h$. In this situation, $H$ is a horizontal sub-rectangle of $(f,\mathcal{R})$ if and only if $h(H)$ is a horizontal sub-rectangle of $(g,h(\mathcal{R}))$. Similarly, $V$ is a vertical sub-rectangle of $(f,\mathcal{R})$ if and only if $h(V)$ is a vertical sub-rectangle of $(g,h(\mathcal{R}))$.
\end{lemm}

\begin{proof}
	
	Observe that $h(\overset{o}{R_i})=\overset{o}{h(R_i)}$. Therefore, $C$ is a connected component of $\overset{o}{R_i}\cap f^{\pm}(\overset{o}{R_j})$ if and only if $h(C)$ is a connected component of $\overset{o}{h(R_i)}\cap g^{\pm}(\overset{o}{h(R_j)})$.
	
\end{proof}

\begin{theo}\label{Theo: Conjugated partitions same types}
	Let $f$ and $g$ be generalized pseudo-Anosov homeomorphisms conjugated through a homeomorphism $h$ that preserves the orientation, i.e., $g = h \circ f \circ h^{-1}$. Let $\cR=\{R_i\}_{i=1}^n$ be a geometric Markov partition for $f$, and let $h(\cR)=\{h(R_i)\}_{i=1}^n$ be the geometric Markov partition of $g$ induced by $h$. In this situation, the geometric types of $(g,h(\cR))$ and $(f,\cR)$ are the same.
\end{theo}

\begin{proof}
	Let $T(f,\cR)=(n,\{h_i,v_i\},\Phi:=(\rho,\epsilon))$ be the geometric type of the geometric Markov partition $\cR$ of $f$. Let $T(g,h(\cR))=(n',\{h'_i,v'_i\},\Phi'_T:=(\rho',\epsilon'))$ be the geometric type of the induced Markov partition.  A direct consequence of Lemma  \ref{Lemm: Conjugated partitions same type.} is that: $n=n'$, $h_i=h'_i$ and $v_i=v'_i$.
	
	Let $\{\overline{H^i_j}\}_{j=1}^{h_i}$ be the set of horizontal sub-rectangles of $h(R_i)$, labeled with respect to the induced vertical orientation in $h(R_i)$. Similarly, we define $\{\overline{V^k_l}\}_{l=1}^{v_k}$ as the set of vertical sub-rectangles of $h(R_k)$, labeled with respect to the horizontal orientation induced by $h$ in $h(R_k)$.

	By the we choose the orientations in $h(R_i)$ and in $h(R_k)$ is clear that:
	$$
	h(H^i_j) =\overline{H^i_j} \text{ and } h(V^k_l) =\overline{V^k_l}.
	$$
	Even more, using the conjugacy we have that: if $f(H^i_j)=V^k_l$ then
	$$
	g(\overline{H^i_j})= g(h(H^i_j))=h(f(H^i_j))=h(V^k_l)=\overline{V^k_l}.
	$$
	
	This implies that in the geometric types, $\rho = \rho'$.
	
	Suppose that $\overline{V^k_l} = g(\overline{H^i_j})$, so the latter set is equal to $h \circ f \circ h^{-1}(\overline{H^i_j})$. The homeomorphism $h$ preserves the vertical orientations between $R_i$ and $h(R_i)$, as well as between $R_k$ and $h(R_k)$. Therefore, $f$ sends the positive vertical orientation of $H^i_j$ with respect to $R_i$ to the positive vertical orientation of $V^k_l$ with respect to $R_k$ if and only if $g$ sends the positive vertical orientation of $\overline{H^i_j}$ with respect to $h(R_i)$ to the positive vertical orientation of $\overline{V^k_l}$ with respect to $h(R_k)$. Then it follows that $\epsilon(i,j) = \epsilon'(i,j)$ and our proof is ended.
\end{proof}

\begin{coro}\label{Coro: constant type in orbit}
	For any primitive Markov partition $\mathcal{R}(z,n)$, where $n \geq n(f)$ and $z$ is a first intersection point of $f$, the following equality holds for all $m \in \mathbb{Z}$:
	$$
	T(\mathcal{R}(z,n)) = T(\mathcal{R}(f^m(z),n)).
	$$
\end{coro}

\begin{theo}\label{Theo: finite geometric types}
	For every $n \geq n(f)$, the set $\cT(f,n) := \{T(\cR) : \cR \in \cM(f,n)\}$  of the so-called \emph{primitive geometric types} $f$ of order $n$ is finite.
\end{theo}

\begin{proof}
	In view of Corollary \ref{Coro: Finite orbits of primitive Markov partitions}, for all $n \geq n(f)$, there exists a finite number of orbits of primitive Markov partitions of $f$ of order $n$. Moreover, the geometric type is constant within each orbit of such a partition. Therefore, for all $n \geq n(f)$, there are only a finite number of distinct geometric types corresponding to the primitive geometric Markov  partitions of order $n$.  
\end{proof}

Of course, within all the primitive geometric types, there is a distinguished family that, in some sense, has the least complexity with respect to its compatibility coefficient.

\begin{defi}\label{defi: Canonical types}
	The elements of the set $\cT(f, n(f))$ are called the canonical geometric types of $f$.
\end{defi}

In fact, in Corollary \ref{Coro: n(f) conjugacy invariant}, we have proved that if $f$ and $g$ are topologically conjugate, then $n(f)=n(g)$. By Theorem  \ref{Theo: Conjugated partitions same types}, we deduce that for every $n\geq n(f)=n(g)$, the set of primitive types of $f$ of order $n$ coincides with the set of primitive geometric types of $g$ of order $n$.

\begin{coro}\label{Coro: conjugated implies same cannonical types}
	If $f$ and $g$ are conjugate generalized pseudo-Anosov homeomorphisms,then,  for all $m\geq n(f)=n(g)$, $\cT(f,m)=\cT(g,m)$ and their canonical types are the same, i.e., $\cT(f,n(f))=\cT(g,n(g))$.
\end{coro}

\subsection*{Acknowledgements}

Part of this work was carried out as the author's PhD thesis at the \emph{Université de Bourgogne} under the supervision of Christian Bonatti, whose guidance is gratefully acknowledged.  
The author was supported during his doctoral studies by the program \emph{Becas al Extranjero Convenios GOBIERNO FRANCÉS 2019--1} (CONACYT).  
A preliminary version of this article was written at the Institute of Mathematics, UNAM, Oaxaca, with support from Lara Bossinger through the grant \emph{DGAPA--PAPIIT IA100724}.  
The author also thanks Ferran Valdez for valuable discussions and comments on the final version of the manuscript.

\end{document}